\documentclass[11pt, reqno]{amsart}
\usepackage[utf8x]{inputenc}  

\usepackage{enumerate}
\usepackage{setspace}
\usepackage[left=3.2cm, right=3.2cm, bottom=4cm]{geometry}               

\usepackage{graphicx} 
\usepackage{pgf,tikz}
\usetikzlibrary{arrows, shapes}
\usepackage{amssymb}
\usepackage{pdfsync}
\usepackage{mathrsfs}
\usepackage{hyperref} 
\usepackage{verbatim} 
\usepackage{epstopdf}
\usepackage{bbm} 
\usepackage[colorinlistoftodos,prependcaption,textsize=tiny]{todonotes}
\usepackage{xargs}
\usepackage[mathscr]{euscript}

\def\R{\mathbb{R}}

\def\Z{\mathbb{Z}}

\newcommand{\scriptC}{\mathcal{C}}
\newcommand{\scriptD}{\mathcal{D}}
\newcommand{\scriptE}{\mathcal{E}}

\newcommand{\scriptR}{\mathcal{R}}

\newcommand{\fP}{\mathcal{P}}
\newcommand{\ang}{\operatorname{ang}}

\newcommand{\sD}{\mathscr{D}}
\newcommand{\fE}{\mathcal{E}}
\newcommand{\fH}{\mathcal{H}}
\newcommand{\fK}{\mathcal{K}}
\newcommand{\fI}{\mathcal{I}}

\newcommand{\bK}{\boldsymbol{K}}
\newcommand{\dist}{\operatorname{dist}}

\newcommand{\qtq}[1]{\quad\text{#1}\quad}

\def\supp{{\rm supp}}

\newcommandx{\emanuel}[2][1=]{\todo[linecolor=green,backgroundcolor=green!25,bordercolor=black,#1]{#2}}

\newcommandx{\diogo}[2][1=]{\todo[linecolor=orange,backgroundcolor=orange!25,bordercolor=orange,#1]{#2}}

\newcommandx{\mateus}[2][1=]{\todo[linecolor=blue,backgroundcolor=blue!25,bordercolor=blue,#1]{#2}}

\newcommandx{\danger}[2][1=]{\todo[linecolor=red,backgroundcolor=red!25,bordercolor=blue,#1]{#2}}

\renewcommand{\d}{\text{\rm d}}

\newcommand{\one}{\mathbbm 1}
\newcommand{\eps}{\varepsilon}

\newcommand{\jp}[1]{\langle{#1}\rangle}
\newcommand{\llangle}{{\langle\!\langle}}
\newcommand{\rrangle}{{\rangle\!\rangle}}
\newcommand{\hjp}[1]{\llangle{#1}\rrangle}

\numberwithin{equation}{section}
\allowdisplaybreaks

\newtheorem{theorem}{Theorem}[section]
\newtheorem{proposition}[theorem]{Proposition}

\newtheorem{lemma}[theorem]{Lemma}

\theoremstyle{definition}
\newtheorem{definition}[theorem]{Definition}

\theoremstyle{remark}
\newtheorem{remark}[theorem]{Remark}

\numberwithin{equation}{section}

\date{\today}                                           

\title{Restriction inequalities for the hyperbolic hyperboloid}
\author{Benjamin Baker Bruce}
\author{Diogo Oliveira e Silva}
\author{Betsy Stovall}

\address{University of Wisconsin--Madison\\ 
480 Lincoln Drive\\ 
Madison, WI 53706\\
USA.}
\email{bbruce@math.wisc.edu}
\email{stovall@math.wisc.edu}

\address{
        School of Mathematics\\
        University of Birmingham\\
        B15 2TT, England, UK.}
\email{d.oliveiraesilva@bham.ac.uk}

\begin{document}

\begin{abstract}
In this article we establish new inequalities, both conditional and unconditional, for the restriction problem associated to the hyperbolic, or one-sheeted, hyperboloid in three dimensions, endowed with a Lorentz-invariant measure.  These inequalities are unconditional (and optimal) in the bilinear range $q > \tfrac{10}3$.
\end{abstract}

\maketitle
\setcounter{tocdepth}{1}
\tableofcontents

\section{Introduction}

This article concerns the boundedness of the Fourier restriction operator associated to the hyperbolic, or one-sheeted, hyperboloid in $\R^{1+2}$,
$$
\Gamma:=\{(\tau,\xi) \in \R^{1+2} : 1+\tau^2 = |\xi|^2\}.  
$$
This surface is invariant under the Lorentz transformations 
\begin{equation} \label{E:Lorentz}
L_\nu:(\tau,\xi) \mapsto (\jp{\nu}\tau-\nu\cdot\xi,\xi^\perp+\jp{\nu}\xi^{\parallel}-\nu\tau),  \qquad \nu \in \R^2,
\end{equation}
where $\jp{\nu}:=\sqrt{1+|\nu|^2}$ and $\xi^\perp,\xi^\parallel$ are the perpendicular and parallel components with respect to $\nu$.  We endow the surface with the unique (up to scalar multiples) Lorentz-invariant measure, which coincides with what is known as the affine surface measure, 
$$
\int_\Gamma f \, \d\sigma = \int_{\{|\xi|>1\}} (f(-\hjp{\xi},\xi) + f(\hjp{\xi},\xi)) \, \tfrac{\d\xi}{\hjp{\xi}}, \qtq{where} \hjp{\xi} := \sqrt{|\xi|^2-1}, \: |\xi| \geq 1.
$$

Various geometric features of this surface make it potentially interesting from the perspective of Fourier restriction/extension.  Though the Gaussian curvature is nonvanishing, the principal curvatures have different signs, which presents challenges at all scales because the restriction theory for hyperbolic surfaces is much less well-developed than that for elliptic surfaces.  One of the main contributions of the present article is an adaptation of the techniques of \cite{Lee neg, BS neg, Vargas neg} to establish unconditional, global restriction inequalities in the bilinear range.  In particular, we establish the first extension inequalities on the parabolic scaling line $q=2p'$ beyond the Stein--Tomas range  (i.e.\ with $p>2$)  for any negatively curved surface that is not the hyperbolic paraboloid.  The above-mentioned techniques are directly applicable in the low frequency region $\{|\xi| \lesssim 1\}$, but at high frequencies, the surface is asymptotic to the cone, presenting some additional complications.  In this region, we use conic decoupling and interpolation with bilinear inequalities to prove a conditional result that boosts local restriction inequalities on the low frequency region to global ones in a range that is non-optimal but, nevertheless, offers the possibility of improvement over that obtainable directly from bilinear restriction.  Our explorations of the conic region also suggest possible future applications of some (surprisingly, still open) questions about the restriction operator associated to the cone in $1+2$ dimensions.  

We turn now to statements of our main results, given in terms of the Fourier extension operator $\scriptE f:=\widehat{f\d\sigma}$, and its local version $\scriptE_0 f := \scriptE(\one_{\{|\xi| \lesssim 1\}} f)$.  We say that $\scriptR^*(p \to q)$ holds if there exists a universal constant $C$ such that $\|\scriptE f\|_{L^q(\R^3)} \leq C\|f\|_{L^p(\Gamma;\d\sigma)}$, for all $f \in C^\infty_{\rm{cpct}}(\R^3)$; we say that $\scriptR^*_0(p \to q)$ holds when the analogous statement holds with $\scriptE_0$ in place of $\scriptE$.  

\begin{theorem}\label{T:pos}
For $(p,q) \neq (4,4)$ obeying $2p' \leq q \leq 3p'$, $q \geq p$, and $q > \tfrac{10}3$, $\scriptR^*(p \to q)$ holds.  Moreover, for $3<q_0 < \tfrac{10}3$, $\scriptR^*_0((\frac{q_0}{2})' \to q_0)$ implies $\scriptR^*(p \to q)$ for all exponent pairs obeying $q_0 < q \leq \tfrac{10}{3}$,  $(\tfrac q2)' \leq p \leq q$, and   
$$
\frac{1}{p} > \frac{2}{5}\cdot\frac{1/q - 3/10}{1/q_0 - 3/10} + \frac{1}{10}.
$$
\end{theorem}

In particular, the first author proved in \cite{BenHypHyp} (see also Remark \ref{scaling line}) that $\scriptR_0^*((\tfrac{q_0}2)'\rightarrow q_0)$ holds for $q_0 > 3.25$, and so our conditional result implies that $\scriptR^*(p \to q)$ holds for $q \leq \tfrac{10}3$, $(\tfrac q2)' \leq p \leq q$, and 
$$
\frac1p > \frac{52}q - \frac{31}2.
$$
(The upper line segment of this region has endpoints $(\tfrac1p,\tfrac1q) = (
\tfrac{31}{102},\tfrac{31}{102})$ and $(\tfrac{7}{18},\tfrac{11}{36})$.)  Because of the loss in the range of $q$, we expect that our conditionality in Theorem \ref{T:pos} is not optimal.  However, improvements to the range of $L^p \times L^p \to L^q$ bilinear extension inequalities for the cone in $\R^3$ may suggest a means of improving the range in our conditional result.

\begin{figure}
\begin{tikzpicture}[scale=13]

\tikzstyle{opendot} = [shape=circle, draw, inner sep = 0pt, minimum size = 1mm]
\tikzstyle{closeddot} = [shape=circle, fill, inner sep = 0pt, minimum size = 1mm]

\node (origin) at (0,0){};
\node (x axis endpoint) at (1.03,0)[label=right:$\frac{1}{p}$]{}; 
\node (y axis endpoint) at (0, .53)[label=above:$\frac{1}{q}$]{};

\node (bilinear line l) at (0,3/10)[label=left:$\frac{3}{10}$]{}; 
\node (bilinear line r) at (2/5,3/10){};

\node (A) at (1,0)[label=below:1]{};
\node (B) at (1/4,1/4)[opendot, label=left:{$(\frac{1}{4},\frac{1}{4})$}]{};
\node (C) at (7/22,7/22) [opendot, label={[text depth=1.5ex] left:{$(\frac{7}{22},\frac{7}{22})$}}]{};
\node (D) at (1/3,1/3) [opendot, label=above:{$(\frac{1}{3},\frac{1}{3})$}]{};
\node (E) at (5/14,9/28) [opendot, label={[text depth = 2ex]right:{$(\frac{5}{14},\frac{9}{28})$}}]{};

\draw[->] (origin.center) -- (x axis endpoint.center);
\draw[->] (origin.center) -- (y axis endpoint.center);

\draw[red] (A.center) -- (B);
\draw[red] (B) -- (C);
\draw[densely dotted, red] (C) -- (E);
\draw[red] (E) -- (A.center);
\draw[densely dotted] (C) -- (D);
\draw[densely dotted] (D) -- (E);

\draw[densely dotted, blue] (bilinear line l.center) -- (bilinear line r.center);

\end{tikzpicture}

\caption{By Theorem \ref{T:pos}, the full restriction conjecture for the low frequency region would imply global restriction estimates for exponent pairs $(p^{-1},q^{-1})$ within the red quadrilateral. Unconditional estimates hold in the bilinear range $q > \frac{10}{3}$.}
\end{figure}
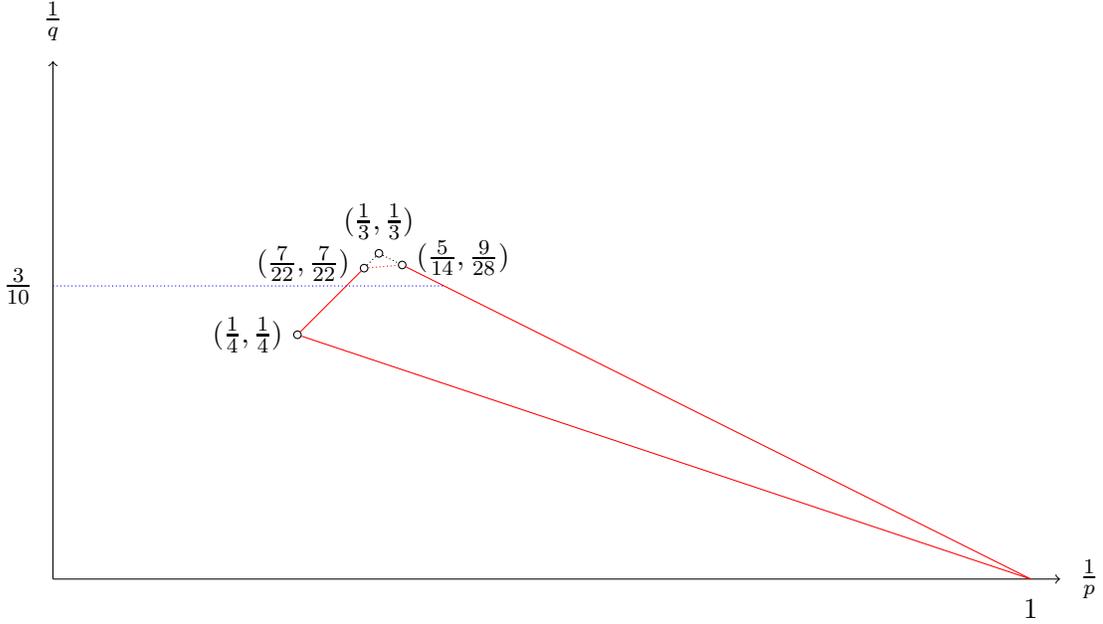

By contrast with Theorem~\ref{T:pos}, we note the following negative result.  

\begin{proposition} \label{P:neg}
For $(p,q) \in\{ (3,3), (4,4)\}$ and for $(p^{-1},q^{-1})$ lying  outside of the triangle
$$
T:=\{(p^{-1},q^{-1}): 2p' \leq q \leq 3p', \: q \geq p\},
$$
$\scriptR^*(p \to q)$ fails.  
\end{proposition}

Though the argument is fairly simple, the authors had not expected to find any exponent pairs along the diagonal $q=p$ at which $\scriptR^*(p \to q)$ holds.  The Kakeya-like example of \cite{BCSS} rules out even a restricted weak type inequality at the endpoint $(3,3)$, but we are not able to exclude the possibility that some weaker inequality might be valid at the endpoint $(4,4)$, and, in fact, the analogous question for the extension operator associated to the cone also seems to be open.

\subsection*{Overview} We prove the negative result, Proposition~\ref{P:neg}, in Section~\ref{S:neg} via familiar Knapp and Kakeya-like examples.  In Section~\ref{S:classical}, we give a brief, self-contained proof of Theorem~\ref{T:pos} in the classical range, $q>4$.  We also record a family of $L^2$-based mixed norm (Strichartz) inequalities which will be useful later on.  In moving beyond the classical range, we begin with our unconditional result:  that $\scriptR_0^*(p \to q)$ holds in the bilinear range $q \geq 2 p'$ and $q>\tfrac{10}3$.  This argument will occupy Section~\ref{S:unconditional bounds in zero region}, in which we establish the bilinear-to-linear deduction for this surface, and Section~\ref{S:bilinear tiles}, in which we prove an $L^2 \times L^2 \to L^{q/2}$ bilinear extension theorem for appropriately separated ``tiles''.  The geometry of the surface, namely the double ruling, plays a critical role, because it enables us to define a bi-parameter family of ``tiles'' that is quite close to that which naturally arises in the case of the hyperbolic paraboloid.  In Section~\ref{S:bilinear conic}, we note that, via a Lorentz boost, $\scriptR_0^*(p \to q)$ implies bounds on unit width ``sectors'' at scale $N$, and we use bilinear extension estimates (similar to those for the cone) to deduce from our unconditional result uniform bounds for the extension from dyadic frusta
$$
\Gamma_N :=\{(\tau,\xi) \in \Gamma : |\xi| \sim 2^N\}.
$$
In Section~\ref{S:decoupling}, we use conic decoupling to extend the deduction in Section~\ref{S:bilinear conic} and obtain a conditional result in a larger (but likely non-optimal) range.  Finally, in Section~\ref{S:annuli to global}, we prove that uniform estimates for the extension from dyadic frusta imply global bounds for $\scriptE$.   

\subsection*{Notation} We will use throughout the standard notation $A \lesssim B$ to mean that $A \leq CB$, for a constant $C$ that is allowed to depend on the Lebesgue exponents in question and also, in the case of conditional results, on assumed finite bounds on the operator norms of the extension operator.  The expression $A \sim B$ means $A \lesssim B$ and $B \lesssim A$.

\section{The negative result:  Proof of Proposition~\ref{P:neg}} \label{S:neg}

\begin{proof}[Proof of necessity of $q \geq 2p'$] We apply the usual Knapp example.  Indeed, if $f_\delta$ is a smooth bump function of radius $0<\delta < 1$ on $\Gamma$, centered at $(0,0,1)$, then $\|f_\delta\|_p \sim \delta^{\frac2p}$, while $|\scriptE f_\delta| \sim \delta^2$ on a tube of length $c\delta^{-2}$ and width $c\delta^{-1}$, yielding $\|\scriptE f_\delta\|_q \gtrsim \delta^{2-\frac4q}$.  
In more detail, set $\phi(\zeta):=\sqrt{1+\zeta_1^2-\zeta_2^2}$, and note that 
$\Gamma\supseteq\{(\zeta_1,\zeta_2,\phi(\zeta)):\zeta_2^2<\frac14(1+\zeta_1^2)\}$.
Affine surface measure on $\Gamma$ is expressed in these coordinates by $\d\sigma(\zeta)=\frac{\d\zeta}{\phi(\zeta)}$. 
Given sufficiently small $\delta>0$, let 
\begin{equation}\label{eq:Cap}
    \scriptC_\delta=\{(\zeta,\phi(\zeta))\in\Gamma: |\zeta|\leq\delta\}
\end{equation}
denote the cap on $\Gamma$ of radius $\delta$ centered at $(0,0,1)$, and consider its indicator function $\mathbbm 1_\delta=\mathbbm 1_{\scriptC_\delta}$. Then
\begin{equation}\label{eq:Knapp1}
\|\mathbbm 1_\delta\|_{L^p(\Gamma;\d\sigma)}=\sigma(\scriptC_\delta)^{\frac1p}\sim\delta^{\frac2p},
\end{equation}
whereas
\[\scriptE(\mathbbm 1_\delta)(t,x)=e^{ix_2}\int_{|\zeta|\leq\delta}e^{i(t,x_1,x_2)\cdot(\zeta_1,\zeta_2,\phi(\zeta)-1)}\frac{\d\zeta}{\phi(\zeta)}.\]
Since $|\zeta_1|,|\zeta_2|\leq\delta$, it follows that $|\phi(\zeta)-1|\leq C\delta^2$.
Consequently, if $|t|,|x_1|\leq C_1^{-1}\delta^{-1}$ and $|x_2|\leq C_1^{-1}\delta^{-2}$, and $C_1$ is sufficiently large, then
\begin{align*}
    |\scriptE(\mathbbm 1_\delta)(t,x)|
    &=\left|\int_{|\zeta|\leq\delta}e^{i(t,x_1,x_2)\cdot(\zeta_1,\zeta_2,\phi(\zeta)-1)}\frac{\d\zeta}{\phi(\zeta)}\right|\\
    &\geq\int_{|\zeta|\leq\delta}\cos(t\zeta_1+x_1\zeta_2+x_2(\phi(\zeta)-1))\frac{\d\zeta}{\phi(\zeta)}
    \geq\frac{\delta^2}{2},
\end{align*}
and therefore
\begin{equation}\label{eq:Knapp2}
    \|\scriptE(\mathbbm 1_\delta)\|_{L^q(\R^3)}^q=\int_{\R^3} |\scriptE(\mathbbm 1_\delta)(t,x)|^q \d t \d x\gtrsim\delta^{2q}(\delta^{-1}\delta^{-1}\delta^{-2})=\delta^{2q-4}.
\end{equation}
If $\scriptE:L^p(\Gamma;\d\sigma)\to L^q(\R^3)$ defines a bounded operator, then from \eqref{eq:Knapp1} and \eqref{eq:Knapp2} it follows that
\[\delta^{2-\frac{4}{q}}\lesssim\|\scriptE(\mathbbm 1_\delta)\|_{L^q(\R^3)}\lesssim\|\mathbbm 1_\delta\|_{L^p(\Gamma;\d\sigma)}\sim\delta^{\frac{2}{p}}.\]
Sending $\delta \searrow 0$ implies $\frac2p \leq 2-\frac4q$, as claimed.  
\end{proof}

\begin{proof}[Proof of necessity of $q \leq 3p'$] We apply a conic Knapp example. 
Details are analogous to the previous paragraph, so we shall be brief.
For $r>0$ sufficiently small and $\lambda > 0$ sufficiently large, consider the set 
\begin{equation}\label{eq:Gammarla}
\Gamma_{r,\lambda} := \left\{(\tau,\xi) \in \Gamma : \tau \sim \lambda, \: |\tfrac{\xi}{|\xi|}-e_1| < r\right\},
\end{equation}
where $e_1 \in \R^2$ denotes the first coordinate vector.  Let $f_{r,\lambda}$ be a smooth bump function adapted to $\Gamma_{r,\lambda}$.  Then $\|f_{r,\lambda}\|_p \sim (\lambda r)^{\frac1p}$, and $|\scriptE f_{r,\lambda}| \sim \lambda r$ on a slab of length $c(\lambda r^2)^{-1}$ (perpendicular to $\Gamma_{r,\lambda}$), width $c(\lambda r)^{-1}$ (tangent to $\Gamma_{r,\lambda}$ in the angular direction), and mini width $c\lambda^{-1}$ (tangent to $\Gamma_{r,\lambda}$ in the radial direction).  Thus $\|\scriptE f_{r,\lambda}\|_q \gtrsim (\lambda r)^{1-\frac 3q}$.  Holding $r$ fixed and sending $\lambda \to \infty$ yields $\frac1p \geq 1-\frac3q$, as claimed.  
\end{proof}

\begin{proof}[Proof of necessity of $q \geq p$] We apply the standard example of summing many disjoint, highly modulated caps whose $L^p$ and $L^q$ norms are all comparable to one another.  For $k \geq 1$, consider the functions $g_k(\tau,\xi):=e^{i(t_k,x_k)(\tau,\xi)}f_{2^{-k},2^k}(\tau,\xi)$, with the $(t_k,x_k)$ to be determined. 
Here, $f_{2^{-k},2^k}$ is a smooth bump function adapted to $\Gamma_{2^{-k},2^k}$; recall \eqref{eq:Gammarla}.
 Then the previous paragraph implies that $\|g_k\|_p \sim 1$, while $\|\scriptE g_k\|_q \gtrsim 1$.  For $(t_k,x_k)$ sufficiently widely separated, we then have that $\|\sum_{k=1}^N g_k\|_p \sim N^{\frac1p}$ and $\|\scriptE (\sum_{k=1}^N g_k)\|_q \gtrsim N^{\frac1q}$, from which we see the necessity of $q \geq p$. 
\end{proof}

\begin{proof}[Proof of necessity of $(p,q) \neq (3,3), (4,4)$]  This follows by either using parabolic, resp.\ conic scaling, Fatou's lemma, and the fact that the corresponding inequalities do not hold for the hyperbolic paraboloid nor for the cone, or by directly using stationary phase.  At the endpoint $(3,3)$, the Kakeya-like example of \cite{BCSS} rules out the possibility of even a restricted weak type inequality, but the authors have not been able to exclude the possibility that weaker inequalities might hold at the endpoint $(4,4)$.   
\end{proof}

\section{Proof in the classical range $q>4$} \label{S:classical}
 
We will use the mixed-norm Strichartz inequality 
\begin{equation}\label{mixednormStr}
\begin{gathered}
\|\scriptE f\|_{L^r_tL^s_x(\R^{1+2})} \lesssim \|\hjp{\xi}^{\frac1r-\frac1s}f\|_{L^2(\Gamma;\d\sigma)}, \\ 
2 \leq r,s; \quad s < \infty; \quad \frac 2r + \frac{1+\theta}{s} = \frac{1+\theta}{2},\text{ for some }\theta\in[0,1].
\end{gathered}
\end{equation}
This classical estimate follows from a straightforward modification of the methods in \cite{KO11}.  
As \eqref{mixednormStr} implies boundedness of $\scriptE$ in the range $p=2$, $4 \leq q \leq 6$, by interpolation it will suffice to restrict attention to the conic line $q=3p'$.

\begin{proposition} \label{P:q=3p'}
Theorem~\ref{T:pos} holds on the line $q=3p'$.  
\end{proposition}

This result can be proved by slicing (see \cite{Drury-Guo, Nicola}).
For the convenience of the reader, we include some details.
\begin{proof}[Proof of Proposition \ref{P:q=3p'} via slicing]
By interpolation, it suffices to restrict attention to $2\leq p<4$ and $q=3p'$. Since $4<q\leq 6$, it follows that $p<q$.
In polar coordinates, we have that
\begin{align*}
    \scriptE f(t,x) =\int_{|\xi|>1} f(\xi)e^{i(t,x)\cdot(\hjp{\xi},\xi)}\frac{\d\xi}{\hjp{\xi}} =\int_1^\infty \left(\int_{\mathbb{S}^1} f(r\omega) e^{irx\cdot\omega} \d\gamma(\omega)\right) e^{it\sqrt{r^2-1}}\frac{r}{\sqrt{r^2-1}} \d r,
\end{align*}
where $\gamma$ denotes the usual arc length measure on the unit circle $\mathbb{S}^1\subset\R^2$.
Changing variables $\sqrt{r^2-1}=s$, and applying the Lorentz space version of the Hausdorff--Young inequality together with Minkowski's integral inequality,\footnote{These are valid moves since $\max\{q,q'\}=q>2$.} yields
\begin{multline}\label{eq:LqEf}
    \|\scriptE f\|_{L^q(\R^3)}\lesssim \left\|\int_{\mathbb{S}^1}e^{i\sqrt{1+s^2}x\cdot\omega} f(\sqrt{1+s^2}\omega) \d\gamma(\omega)\right\|_{L_x^q L_s^{q',q}}\\
\lesssim\left\|\int_{\mathbb{S}^1}e^{i\sqrt{1+s^2}x\cdot\omega} f(\sqrt{1+s^2}\omega) \d\gamma(\omega)\right\|_{L_s^{q',q} L_x^q }.
\end{multline}
A further change of variables $y=\sqrt{1+s^2}x$ allows us to estimate the inner norm on the right-hand side of \eqref{eq:LqEf} as follows:
\begin{align*}
\left\|\int_{\mathbb{S}^1}e^{i\sqrt{1+s^2}x\cdot\omega} f(\sqrt{1+s^2}\omega) \d\gamma(\omega)\right\|_{L_x^q }   
&=(1+s^2)^{-1/q}
\left\|\int_{\mathbb{S}^1}e^{iy\cdot\omega} f(\sqrt{1+s^2}\omega) \d\gamma(\omega)\right\|_{L_y^q}\\
&\lesssim (1+s^2)^{-1/q} \|f(\sqrt{1+s^2}\cdot)\|_{L^p(\mathbb{S}^1)}
\end{align*}
where the latter estimate follows from the $L^p(\mathbb{S}^1;\d\gamma)\to L^q(\R^2)$ adjoint restriction inequality on the unit circle $\mathbb{S}^1$; see \cite{Zyg74}.
Going back to \eqref{eq:LqEf}, we then have that
\begin{align*}  
\|\scriptE f\|_{L^q(\R^3)}
&\lesssim\left\|(1+s^2)^{-1/q}\|f(\sqrt{1+s^2}\cdot)\|_{L^p(\mathbb{S}^1)}\right\|_{L_s^{q',q}}\\
&\lesssim\left\|(1+s^2)^{-1/q}\|f(\sqrt{1+s^2}\cdot)\|_{L^p(\mathbb{S}^1)}\right\|_{L_s^{q',p}}
\end{align*}
where the latter estimate holds since $p<q$.
Denote
\[F(s):=(1+s^2)^{-1/q} \text{ and }G(s):=\|f(\sqrt{1+s^2}\cdot)\|_{L^p(\mathbb{S}^1)},\]
and let $\alpha>0$ be such that $\frac 1{q'}=\frac1{\alpha}+\frac1p$.
Then the Lorentz space version of H\"older's inequality implies
\[\|\scriptE f\|_{L^q(\R^3)}\lesssim \|F\|_{L_s^{\alpha,\infty}} \|G\|_{L_s^{p,p}}.\]
To check that $F\in L_s^{\alpha,\infty}$, simply note that $2\alpha=q$ since $\frac 1{q'}=\frac1{\alpha}+\frac1p$ and $q=3p'$.
Finally, reverting back to the original variable $r=\sqrt{1+s^2}$, we see that
\begin{align*}
\|G\|_{L_s^{p,p}}^p
=\|G\|_{L_s^{p}}^p
=\int_0^\infty \|f(\sqrt{1+s^2}\cdot)\|_{L^p(\mathbb{S}^1)}^p \d s\\
=\int_1^\infty \int_{\mathbb{S}^1} |f(r\omega)|^p \d\gamma(\omega) \frac{r}{\sqrt{r^2-1}} \d r
=\|f\|_{L^p(\Gamma;\d\sigma)}^p.
\end{align*}
This shows the boundedness of the operator $\scriptE:L^p(\Gamma;\d\sigma)\to L^q(\R^3)$ whenever $2\leq p<4$ and $q=3p'$, as desired. 
\end{proof}

\section{Unconditional bounds at low frequencies in the bilinear range}\label{S:unconditional bounds in zero region}

We turn now to the heart of the article, the proof of Theorem~\ref{T:pos} beyond the classical range.  We begin by bounding the low-frequency extension operator $\scriptE_0$ in the bilinear range ($q >\tfrac{10}3$, $p \geq (\tfrac q2)'$), which will occupy the next two sections.  The companion article, \cite{BenHypHyp}, bounds $\scriptE_0$ in the polynomial range ($q > 3.25$, $p \geq (\frac{q}{2})'$), except on the scaling line $p = (\frac{q}{2})'$. Utilizing the results of this and the next section, we can extend the strictly local ($p > (\frac{q}{2})'$) inequalities of \cite{BenHypHyp} to the scaling line.  We sketch this argument in Section \ref{S:bilinear tiles}; see Remark \ref{scaling line}.

It will suffice to prove extension estimates for a small region contained in a rotated version of the hyperboloid.  Let
\begin{align}\label{param hyp}
\Sigma := \Big\{\Big(\sqrt{1+\xi_1^2-\xi_2^2},\xi\Big) \in \R \times \R^2 : |\xi| \leq \frac{1}{2}\Big\},
\end{align}
and let $U$ be a small neighborhood of the origin that we will choose.  We will consider the subset of $\Sigma$ that lies above $U$.   Abusing notation, we define the extension operator $\fE_0$ by
\begin{align*}
\fE_0 f(t,x) := \int_{U} e^{i(t,x)\cdot(\sqrt{1+\xi_1^2-\xi_2^2},\xi)}f(\xi)\d\xi.
\end{align*}
This definition of $\fE_0$ is not quite the same as the one given in Section 1; however, the two operators obey the same range of $L^p \rightarrow L^q$ estimates, as one can see by using the triangle inequality and symmetries of the operator.  We aim to prove the following result.

\begin{theorem}\label{linear est zero region}
If $q > 10/3$, then $\|\fE_0 f\|_q \lesssim \|f\|_{(q/2)'}$ for all $f \in L^{(q/2)'}(U)$.
\end{theorem}

Our starting point will be the $L^2$-based bilinear theory for $\Sigma$, which we will obtain by rescaling a result of Lee \cite{Lee neg}.

\subsection{Related tiles}

To state the bilinear estimate, we must first define ``related tiles", i.e.\@ pairs of subsets of $U$ adapted to the transversality conditions that arise in the bilinear method.  Here, the geometry of the hyperbolic hyperboloid will play a distinguished role, particularly the doubly ruled structure.  Given $(\tau,\xi) \in \Sigma$, the rulings of $\Sigma$ that contain $(\tau,\xi)$ are parametrized by the formulae
\begin{align*}
\ell_{(\tau,\xi)}^\pm(t) := (\tau,\xi) + t(\xi_1\tau \mp \xi_2,1+\xi_1^2, \xi_1\xi_2 \pm \tau),
\end{align*}
and their projections to the spatial coordinates are given by
\begin{align}\label{proj4.2}
\ell^\pm_\xi(t) := \xi + t\Big(1+\xi_1^2, \xi_1\xi_2 \pm \sqrt{1+\xi_1^2-\xi_2^2}\Big).
\end{align}

Fix an integer $n \geq 10$, and set $I := [-2^{-n},2^{-n})$ and $Q := I \times I$.  Moreover, let $D := D(0,1/{10})\subseteq\R^2$ denote the open disc of radius $\frac1{10}$ centered at the origin.  Define maps $\Phi : Q \rightarrow \R^2$  and $\pi^\pm : D \rightarrow \R$ by
\begin{align*}
\Phi(\zeta) := \frac{(\zeta_1\sqrt{1+\zeta_2^2} + \zeta_2\sqrt{1+\zeta_1^2}, \zeta_2-\zeta_1)}{\sqrt{1+\zeta_1^2} + \sqrt{1+\zeta_2^2}}
\end{align*}
and
\begin{align}\label{eq:defpi}
\pi^\pm(\xi) := \xi_1 - \frac{\xi_2(1+\xi_1^2)}{\xi_1\xi_2 \pm \sqrt{1+\xi_1^2-\xi_2^2}}.
\end{align}
Then (possibly after increasing $n$) $\Phi$ is a diffeomorphism satisfying $\frac13 \leq \det\nabla\Phi \leq 1$, $\|\Phi\|_{C^1} \leq 3$, and $\Phi(Q) \subseteq D(0,2^{-n+5})$.  Indeed, $\Phi$ can be viewed as a perturbation of the rotation $\zeta \mapsto \frac{1}{2}(\zeta_1 + \zeta_2, \zeta_2-\zeta_1)$.  Likewise, the maps $\pi^\pm$ are submersions satisfying $\|\pi^\pm\|_{C^1} \leq 3$.  We now set $U := \Phi(Q)$.

\begin{lemma}\label{properties of maps}
The maps $\Phi$ and $\pi^\pm$ satisfy the following geometric properties:
\begin{itemize}
\item[(1)]{ $\{\Phi(\zeta)\} = \ell_{(\zeta_1,0)}^+ \cap \ell_{(\zeta_2,0)}^-$ and $(\pi^\pm(\xi),0) \in \ell_\xi^\pm$ for every $\zeta \in Q$ and $\xi \in D$.}
\item[(2)]{The fibers of $\pi^\pm$ are precisely the line segments $\ell_\xi^\pm \cap D$ with $\xi \in D$.}
\item[(3)]{$\Phi^{-1} = (\pi^+ \times \pi^-)|_U$, where $\pi^+ \times \pi^- (\xi) := (\pi^+(\xi),\pi^-(\xi))$.}
\end{itemize}
\end{lemma}
\begin{proof}
Property (1) can be verified by a straightforward calculation.  It is helpful to reparametrize \eqref{proj4.2} so that the second coordinates of $\ell_{(\zeta_1,0)}^+(t)$ and $\ell_{(\zeta_2,0)}^-(t)$ are $t$ and $-t$, respectively.

Property (2) is a consequence of property (1) and the following claim: If $|\eta|,|\eta'| \leq 1/2$ and $\eta' \in \ell_\eta^\pm$, then $\ell_{\eta'}^\pm = \ell_\eta^\pm$. Indeed, assume the claim holds, and let $\xi \in D$ and $c \in \R$ satisfy $\pi^\pm(\xi) = c$. Then $\xi' \in \ell_\xi^\pm \cap D$ implies that $\ell_{\xi'}^\pm$ and $\ell_{\xi}^\pm$ are identical and thus have the same $x$-intercept. Consequently, $(\pi^\pm)^{-1}(c) \supseteq \ell_\xi^\pm \cap D$ by property (1).  If $\tilde{\xi}$ is another point such that $\pi^\pm(\tilde{\xi}) = c$, then applying the claim twice more shows that $\ell_{\tilde{\xi}}^\pm = \ell_{(c,0)}^\pm = \ell_\xi^\pm$.  Thus, $(\pi^\pm)^{-1}(c) = \ell_\xi^\pm \cap D$.  It remains to prove the claim.  Define $F : \Sigma \rightarrow D(0,1/2)$ by $F(\tau,\xi) := \xi$.  Then $F$ is an invertible map such that $F^{-1}(\ell_\xi^\pm \cap D(0,1/2))$ is a ruling of $\Sigma$ for every $|\xi| \leq 1/2$.  Suppose for contradiction that $\ell_{\eta'}^+ \neq \ell_\eta^+$.  Then the lines $\ell_{\eta'}^+,\ell_{\eta'}^-,\ell_\eta^+$ are distinct (as one can easily check) and intersect at $\eta'$, implying that $F^{-1}(\eta')$ belongs to three rulings of $\Sigma$. However, no three rulings of the hyperbolic hyperboloid intersect at a common point.  Thus, we must have $\ell_{\eta'}^+ = \ell_\eta^+$ and, by a similar argument, $\ell_{\eta'}^- = \ell_\eta^-$.

Property (3) is a consequence of properties (1) and (2).
\end{proof}

We also record that
\begin{align}\label{ortho lines}
\angle(\ell_{\eta}^\pm,\R(1,\pm1)) \leq 10^\circ
\end{align}
for all $\eta \in D$; in particular, we always have $\angle(\ell_{\eta}^+, \ell_{\eta'}^-) \geq 70^\circ$.

For each integer $j > n$, let $\fI_j$ denote the set of dyadic intervals of length $2^{-j}$ contained in $I$; that is,
\begin{align*}
\fI_j := \{[m2^{-j},(m+1)2^{-j}) : m \in \Z \cap [-2^{j-n}, 2^{j-n})\}.
\end{align*}
Given $I_j,I_j' \in \fI_j$, we write $I_j \sim I_j'$ if $I_j$ and $I_j'$ are non-adjacent but have adjacent dyadic parents.   

\begin{definition}
A \emph{tile} is any set of the form $\Phi(I_j \times I_k)$ with $(I_j,I_k) \in \fI_j \times \fI_k$ and $j,k > n$.  We denote by $\Theta_{j,k}$ the set of $2^{-j}\times 2^{-k}$ tiles.  Given $\theta,\theta' \in \Theta_{j,k}$, we write $\theta \sim \theta'$, and say that $\theta$ and $\theta'$ are \emph{related}, if $\pi^+(\theta) \sim \pi^+(\theta')$ and $\pi^-(\theta) \sim \pi^-(\theta')$. (Note that if $\theta = \Phi(I_j \times I_k)$, then $\pi^+(\theta) = I_j$ and $\pi^-(\theta) = I_k$.) Finally, given $C > 0$, we define $C\theta := \Phi(C\Phi^{-1}(\theta) \cap Q)$, where $C\Phi^{-1}(\theta)$ is the $C$-fold dilate of the rectangle $\Phi^{-1}(\theta)$ with respect to its center. 
\end{definition}

We can now state the bilinear restriction theorem for related tiles, which we will prove in the next section.
\begin{theorem}\label{T:bilinear tiles}
Let $\theta_1,\theta_2 \in \Theta_{j,k}$ be related tiles.  If $q > 10/3$, then 
\begin{align*}
\|\fE_0 f\fE_0 g\|_{q/2} \lesssim 2^{(j+k)(\frac{4}{q}-1)}\|f\|_2\|g\|_2
\end{align*}
for all $f \in L^2(\theta_1)$ and $g \in L^2(\theta_2)$.
\end{theorem}

Next, we establish several properties of the tiles $\theta$, most of which are easy consequences of analogous properties of their rectangular counterparts $I_j \times I_k$.

\begin{definition}
Given a measurable set $\Omega \subseteq U$, we call any set of the form $\ell_\xi^\pm \cap \Omega$ with $\xi \in \Omega$ a \emph{$\pi^\pm$-fiber} of $\Omega$.  The \emph{length} of $\ell_\xi^\pm \cap \Omega$ is $\fH^1(\ell_\xi^\pm \cap \Omega)$, where $\fH^1$ denotes the one-dimensional Hausdorff measure. Given an integer $K \geq 0$, we define two sets
\begin{align*}
\Omega(K)^\pm := \{\xi \in \Omega : 2^{-K} \leq \fH^1(\ell_\xi^\pm \cap \Omega) < 2^{-K+1}\},
\end{align*}
and say that $\Omega$ has \emph{constant $\pi^\pm$-fiber length} $2^{-K}$ if $\Omega = \Omega(K)^\pm$.
\end{definition}

\begin{lemma}\label{tiles}
The tiles $\theta$ satisfy the following properties:
\begin{itemize}
\item[(1)]{There exists a set $N \subset U \times U$ of measure zero such that
\begin{align*}
(U \times U) \setminus N = \bigcup_{j,k > n}\,\,\,\bigcup_{\substack{\theta,\theta' \in \Theta_{j,k} :\\ \theta \sim \theta'}} \theta \times \theta';
\end{align*}
moreover, the union is disjoint.}
\item[(2)]{For each pair of related tiles $\theta\sim\theta' \in \Theta_{j,k}$, there exists a rectangle $R_{\theta,\theta'}$ such that $\theta + \theta' \subseteq R_{\theta,\theta'}$ and the collection $\{2R_{\theta,\theta'}\}_{\theta \sim \theta' \in \Theta_{j,k}}$ has bounded overlap.}

\item[(3)]{For every constant $C >0$, the collection of dilates $C\theta$, with $\theta \in \Theta_{j,k}$, has bounded overlap.}
\item[(4)]{For every $\theta \in \Theta_{j,k}$ and constant $C > 0$, we have $|C\theta| \sim 2^{-j-k}$.}
\item[(5)]{For every $\theta \in \Theta_{j,k}$ and constant $C > 0$, the set $C\theta$ has $\pi^+$-fibers and $\pi^-$-fibers of length $O(2^{-k})$ and $O(2^{-j})$, respectively.}
\end{itemize}
\end{lemma}
\begin{proof}
Property (1) is obtained by applying the diffeomorphism $(\zeta,\zeta') \mapsto (\Phi(\zeta),\Phi(\zeta'))$ to the disjoint union
\begin{align*}
(Q \times Q) \setminus M = \bigcup_{j,k > n}\bigcup_{\substack{I_j,I_j' \in \fI_j : I_j \sim I_j'\\ I_k,I_k' \in \fI_k : I_k \sim I_k'}} I_j \times I_k \times I_j' \times I_k'
\end{align*}
with $M := \{(\zeta,\zeta') \in Q \times Q : \zeta_1 = \zeta_1' ~\text{or}~ \zeta_2 = \zeta_2'\}$.

Next, we prove property (2).  We may assume that $j \geq k$.  Fix $\theta \sim \theta' \in \Theta_{j,k}$ and set $c_\theta := \Phi(c_\theta^+,c_\theta^-)$, where $c_\theta^\pm$ is the center of the dyadic interval $\pi^\pm(\theta)$.  We claim that there exists an $O(2^{-j}) \times O(2^{-k})$ rectangle $\tilde{R}_{\theta,\theta'}$ with center $c_\theta$ and major axis $\ell_{c_\theta}^+$ such that $\theta \cup \theta' \subseteq \tilde{R}_{\theta,\theta'}$.  If $\zeta \in \Phi^{-1}(\theta \cup \theta')$, then $|\zeta_1 - c_\theta^+| \lesssim 2^{-j}$ and $|\zeta_2 - c_\theta^-| \lesssim 2^{-k}$, whence
\begin{align*}
\dist(\Phi(\zeta), \ell_{c_\theta}^+) \leq |\Phi(\zeta) - \Phi(c_\theta^+,\zeta_2)|\lesssim \|\Phi\|_{C^1}|\zeta_1 - c_\theta^+| \lesssim 2^{-j}
\end{align*}
and similarly $\dist(\Phi(\zeta),\ell_{c_\theta}^-) \lesssim 2^{-k}$.  Thus, $\theta \cup \theta'$ lies in the intersection of  $O(2^{-j})$- and $O(2^{-k})$-neighborhoods of $\ell_{c_\theta}^+$ and $\ell_{c_\theta}^-$, respectively.  By \eqref{ortho lines}, this intersection is nearly a rectangle; it lies in an $O(2^{-j}) \times O(2^{-k})$ rectangle with center $c_\theta$ and major axis $\ell_{c_\theta}^+$, which we may take as $\tilde{R}_{\theta,\theta'}$.  We can define the rectangle $R_{\theta,\theta'}$ as $R_{\theta,\theta'} := \tilde{R}_{\theta,\theta'}+\tilde{R}_{\theta,\theta'} = 2\tilde{R}_{\theta,\theta'} + c_\theta$.  We need to show that the collection $\{2R_{\theta,\theta'}\}_{\theta\sim\theta' \in \Theta_{j,k}}$ has bounded overlap.  Suppose $\theta_i \sim \theta_i' \in \Theta_{j,k},$ $i = 1,2,$ are such that ${2R}_{\theta_1,\theta_1'} \cap {2R}_{\theta_2,\theta_2'} \neq \emptyset$.  Then there exist points $\xi^i \in {2R}_{\theta_i,\theta_i'} \cap (\ell_{c_{\theta_i}}^+ + c_{\theta_i})$ such that $|\xi^2-\xi^1| \lesssim 2^{-j}$.  Since $c_{\theta_i} \in \ell_{c_{\theta_i}}^+$, it follows that $\xi^i/2 \in \ell_{c_{\theta_i}}^+$.  Moreover, $|\xi^i - 2c_{\theta_i}| \lesssim 2^{-k}$, so if $n$ (and therefore $k$) is sufficiently large, then $\xi^i/2 \in D$.  Since $\pi^+$ is constant on $\ell_{c_{\theta_i}}^+ \cap D$, we see that
\begin{align*}
|c_{\theta_2}^+ - c_{\theta_1}^+| = |\pi^+(\xi^2/2) - \pi^+(\xi^1/2)| \leq \|\pi^+\|_{C^1}|\xi^2/2 - \xi^1/2| \lesssim 2^{-j}. 
\end{align*}
The assumption that ${2R}_{\theta_1,\theta_1'} \cap {2R}_{\theta_2,\theta_2'} \neq \emptyset$ also implies that $|2c_{\theta_2} - 2c_{\theta_1}| \lesssim 2^{-k}$, whence
\begin{align*}
|c_{\theta_2}^- - c_{\theta_1}^-|  \leq \|\pi^-\|_{C^1}|c_{\theta_2}-c_{\theta_1}| \lesssim 2^{-k}.
\end{align*}
Since $c_{\theta_2}^+$ and $c_{\theta_2}^-$ are the centers of dyadic intervals of length $2^{-j}$ and $2^{-k}$, respectively, we have shown the following: If $\theta_1$ is fixed and $2R_{\theta_1,\theta_1'} \cap 2R_{\theta_2,\theta_2'} \neq \emptyset$, then $\theta_2$ must be one of $O(1)$ possible tiles.  Since any tile has at most $O(1)$ relatives, it then follows that the collection $\{2{R}_{\theta,\theta'}\}_{\theta \sim \theta' \in \Theta_{j,k}}$ has bounded overlap.

Property (3) follows from the dilated dyadic rectangles $C(I_j \times I_k)$ having bounded overlap.

Property (4) follows from the change of variables theorem and the fact that $|\det \nabla\Phi| \sim 1$.

Using property (3) in Lemma \ref{properties of maps}, one sees that $\ell_\xi^+ \cap U = \Phi(\{\zeta_1\} \times I)$, where $\xi \in U$ and $\zeta := \Phi^{-1}(\xi)$.  Hence, if $C > 0$, $\theta \in \Theta_{j,k}$, and $\xi \in C\theta$, then $\ell_\xi^+ \cap C\theta = \Phi((\{\zeta_1\} \times I) \cap C\Phi^{-1}(\theta))$.  The line segment $(\{\zeta_1\} \times I) \cap C\Phi^{-1}(\theta)$ has length at most $C2^{-k}$, and thus the bounds on $\nabla \Phi$ imply that $\ell_\xi^+ \cap C\theta$ has length $O(2^{-k})$.  A similar argument applies to the fibers $\ell_\xi^- \cap C\theta$, proving property (5).
\end{proof}

\subsection{Proof of Theorem \ref{linear est zero region}}
Having defined related tiles and shown that they behave like dyadic rectangles, we are ready to prove Theorem \ref{linear est zero region}.  We adapt the argument of the third author in \cite{BS neg},  with the fibers of $\pi^+$ and $\pi^-$ now playing the roles of vertical and horizontal fibers. For the remainder of this section, we will assume that $\tfrac{10}3 < q < 4$.  

The main step is to prove a restricted strong type inequality.  We state and prove the below lemmas for characteristic functions, but the proofs are unchanged if we replace $\one_\Omega$ with a measurable function $f_\Omega$ with $|f_\Omega| \sim \one_\Omega$.  

\begin{proposition}\label{const fiber}
Let $\Omega \subseteq U$ have constant $\pi^+$-fiber length $2^{-K}$ for some integer $K \geq 0$.  Then $\|\fE_0 \one_{\Omega'}\|_q \lesssim |\Omega|^{1-\frac{2}{q}}$ for every measurable set $\Omega' \subseteq \Omega$.
\end{proposition}
\begin{proof} We essentially follow Vargas's argument in \cite{Vargas neg}, but replace dyadic rectangles $I_j \times I_k$ with tiles $\theta$.  Fix a measurable set $\Omega' \subseteq \Omega$.  Using property (1) of Lemma \ref{tiles}, the triangle inequality, almost orthogonality (combining \cite[Lemma 6.1]{TVV} and property (2) of Lemma \ref{tiles}), and finally Theorem \ref{T:bilinear tiles}, we have
\begin{align*}
    \|\fE_0\one_{\Omega'}\|_q^2 &=  \bigg\|\sum_{j,k>n}\sum_{\substack{\theta,\theta' \in \Theta_{j,k}:\\\theta\sim\theta'}}\fE_0(\one_{\Omega'\cap\theta})\fE_0(\one_{\Omega'\cap\theta'})\bigg\|_{q/2}\\
    &\lesssim \sum_{j,k>n}\bigg(\sum_{\substack{\theta,\theta' \in \Theta_{j,k}:\\\theta\sim\theta'}}\|\fE_0(\one_{\Omega' \cap \theta})\fE_0(\one_{\Omega' \cap \theta'})\|_{q/2}^{q/2}\bigg)^\frac{2}{q}\\
    &\lesssim \sum_{j,k>n}2^{(j+k)(\frac{4}{q}-1)}\bigg(\sum_{\substack{\theta,\theta' \in \Theta_{j,k}:\\\theta\sim\theta'}}|\Omega' \cap \theta|^\frac{q}{4}|\Omega' \cap \theta'|^\frac{q}{4}\bigg)^\frac{q}{2}.
\end{align*}
Since $10\theta \supseteq \theta'$ whenever $\theta$ and $\theta'$ are related, each tile has a bounded number of relatives, and the dilates $10\theta$ have bounded overlap, it follows that
\begin{align}\label{first est}
    \|\fE_0\one_{\Omega'}\|_q^2 \lesssim \sum_{j,k>n}2^{(j+k)(\frac{4}{q}-1)}\bigg(\sum_{\theta \in \Theta_{j,k}}|\Omega \cap 10\theta|^\frac{q}{2}\bigg)^\frac{2}{q} \lesssim \sum_{j,k > n}2^{(j+k)(\frac{4}{q}-1)}|\Omega|^\frac{2}{q}\max_{\theta \in \Theta_{j,k}}|\Omega \cap 10\theta|^{1-\frac{2}{q}}.
\end{align}
Let $J$ be an integer such that $|\pi^+(\Omega)| \sim 2^{-J}$.  By the coarea formula, the hypothesis on $\Omega$, and property (5) in Lemma \ref{tiles}, we have $|\Omega| \sim 2^{-J-K}$ and
\begin{align*}
|\Omega \cap 10\theta| &\lesssim |\pi^+(\Omega \cap 10\theta)|\sup_{\xi \in \Omega \cap 10\theta}\fH^1(\ell_\xi^+ \cap \Omega \cap 10\theta)\\
&\lesssim \min\{2^{-J},2^{-j}\}\min\{2^{-K},2^{-k}\},
\end{align*}
for every $\theta \in \Theta_{j,k}$.  Inserting this bound into \eqref{first est} and summing the resulting (four) geometric series produces the required estimate.
\end{proof}

\begin{proposition}\label{decomposition}
Let $\Omega \subseteq U$ have constant $\pi^+$-fiber length $2^{-K}$ for some integer $K \geq 0$, let $J$ be an integer such that $|\Omega| \sim 2^{-J-K}$, and let $\varepsilon$ be the smallest dyadic number such that $\|\fE_0\one_{\Omega'}\|_q \leq \varepsilon^2 |\Omega|^{1-\frac{2}{q}}$ for all measurable sets $\Omega' \subseteq \Omega$.  Up to a set of measure zero, there exists a decomposition
\begin{align*}
\Omega = \bigcup_{0 < \delta \lesssim \varepsilon^{1/4}} \Omega_\delta,
\end{align*}
where the union is taken over dyadic numbers, such that the following properties hold:
\begin{itemize}
\item[(1)]{$\|\fE_0\one_{\Omega'}\|_q \lesssim \delta|\Omega|^{1-2/q}$ for every measurable set $\Omega' \subseteq \Omega_\delta$, and}
\item[(2)]{$\Omega_\delta \subseteq \bigcup_{\theta \in \Theta_\delta}\theta$, where $\Theta_\delta \subseteq \Theta_{J,K}$ with $\#\Theta_\delta \lesssim \delta^{-C_0}$ for some constant $C_0$.}
\end{itemize}
\end{proposition}
\begin{proof}
The construction of the sets $\Omega_\delta$ proceeds in three steps.

\emph{Step 1.} Let $S := \pi^+(\Omega)$.  By the coarea formula, $|S| \sim 2^{-J}$.  Let $\xi_1$ be a Lebesgue point of $S$ and $0 < \eta \leq \varepsilon$ a dyadic number.  Define $I_\eta(\xi_1)$ to be the maximal dyadic interval $I$ such that $\xi_1 \in I$ and
\begin{align}\label{density}
\frac{|I \cap S|}{|I|} \geq \eta^C,
\end{align}
where $C$ is a constant (to be chosen); such an interval exists by the Lebesgue differentiation theorem.  Since we may exclude a set of measure zero in our decomposition, we assume without loss of generality that $S$ is equal to its set of Lebesgue points.  We note that $|I_\eta(\xi_1)| \lesssim \eta^{-C}2^{-J}$.  Let
\begin{align*}
T_\eta := \{\xi_1 \in S : |I_\eta(\xi_1)| \geq \eta^C2^{-J}\},
\end{align*}
and let $S_\varepsilon := T_\varepsilon$ and $S_\eta := T_\eta \setminus T_{2\eta}$ for $\eta < \varepsilon$.  Then every point of $S$ is contained in a unique $S_\eta$.  We set $\Omega_\eta^1 := \Omega \cap (\pi^+)^{-1}(S_\eta)$.

\begin{lemma}\label{first step}
For every $0 < \eta \leq \varepsilon$, the set $\Omega_\eta^1$ is contained in a union of $O(\eta^{-3C})$ tiles in $\Theta_{J,n}$, and for every measurable set $\Omega' \subseteq \Omega_\eta^1$, we have
\begin{align*}
\|\fE_0\one_{\Omega'}\|_q \lesssim \eta^2|\Omega|^{1-\frac{2}{q}}.
\end{align*}
\end{lemma}
\begin{proof}
By its definition, $S_\eta$ is covered by dyadic intervals $I$ of length $|I| \gtrsim \eta^C|S|$, in each of which $S$ has density obeying \eqref{density}.  The density of each such $I$ in $S$ is
\begin{align*}
\frac{|I \cap S|}{|S|} = \frac{|I \cap S|}{|I|}\cdot \frac{|I|}{|S|} \gtrsim \eta^{2C}.
\end{align*}
Thus, a minimal-cardinality covering of $S_\eta$ by these $I$ (which are necessarily pairwise disjoint) has size $O(\eta^{-2C})$.  Additionally, each $I$ satisfies $|I| \lesssim \eta^{-C}2^{-J}$, and thus $S_\eta$ is covered by $O(\eta^{-3C})$ intervals in $\fI_J$.  Consequently,  $\Omega_\eta^1$ is contained in a union of $O(\eta^{-3C})$ tiles in $\Theta_{J,n}$.

We turn to the extension estimate, fixing a measurable set $\Omega' \subseteq \Omega$.  By the definition of $\varepsilon$, we may assume that $\eta < \varepsilon$.  By the same argument that yields \eqref{first est}, we have
\begin{align}\label{first est'}
\|\fE_0\one_{\Omega'}\|_q^2 \lesssim \sum_{j,k > n}2^{(j+k)(\frac{4}{q}-1)}|\Omega|^\frac{2}{q}\max_{\theta \in \Theta_{j,k}}|\Omega' \cap 10\theta|^{1-\frac{2}{q}},
\end{align}
and the coarea formula implies that
\begin{align*}
|\Omega' \cap 10\theta| \lesssim \min\{2^{-J},2^{-j}\}\min\{2^{-K},2^{-k}\},
\end{align*}
for every $\theta \in \Theta_{j,k}$.  If $|j-J| < \frac{C}{4}\log \eta^{-1}$, then the definition of $\Omega_\eta^1$ leads to the stronger estimate
\begin{align}\label{strong est}
|\Omega' \cap 10\theta| \lesssim \eta^\frac{3C}{4}\min\{2^{-J},2^{-j}\}\min\{2^{-K},2^{-k}\}.
\end{align}
Indeed, fix such a $j$.  It suffices to prove \eqref{strong est} with some $\tilde{\theta} \in \Theta_{j-4,k-4}$ in place of $\theta$, since each $\theta$ is contained in a union of four such tiles.  Let $\tilde{\theta} =: \Phi(I_{j-4} \times I_{k-4})$, so that $\pi^+(\tilde{\theta}) = I_{j-4} \in \fI_{j-4}$.  We have
\begin{align*}
|I_{j-4}| \geq 16\eta^\frac{C}{4}2^{-J} \geq (2\eta)^C 2^{-J}
\end{align*}
for $\eta$ sufficiently small (which we may assume).  Suppose that $I_{j-4} \cap S_\eta \neq \emptyset$.  Then there exists $\xi_1 \in I_{j-4}$ such that $\xi_1 \notin T_{2\eta}$, whence
\begin{align*}
|I_{2\eta}(\xi_1)| < (2\eta)^C 2^{-J} \leq |I_{j-4}|.
\end{align*}
Consequently, by the maximality of $I_{2\eta}(\xi_1)$ and the fact that $2^{-j} \leq \eta^{-\frac{C}{4}}2^{-J}$, we have
\begin{align*}
|I_{j-4} \cap S_\eta| \leq |I_{j-4} \cap S| \leq (2\eta)^C|I_{j-4}| = 16(2\eta)^C2^{-j} \lesssim \eta^\frac{3C}{4}\min\{2^{-J},2^{-j}\}.
\end{align*}
Thus, by the coarea formula,
\begin{align*}
|\Omega' \cap \tilde{\theta}| \lesssim |I_{j-4} \cap S_\eta|\min\{2^{-K},2^{-k}\} \lesssim \eta^\frac{3C}{4}\min\{2^{-J},2^{-j}\}\min\{2^{-K},2^{-k}\},
\end{align*}
as claimed.  Inserting this bound into \eqref{first est'} and summing the resulting (eight) geometric series leads to the estimate $\|\fE_0\one_{\Omega'}\|_q \lesssim \eta^{C'}|\Omega|^{1-2/q}$, where $C'$ is a constant determined by $C$.  We can choose $C$ so that $C' = 2$.
\end{proof}

\emph{Step 2.} For dyadic $0 < \eta \leq \varepsilon$ and $0 < \rho \lesssim \eta^{1/4}$, define
\begin{align*}
\Omega_{\eta,\rho}^2 := \{\xi \in \Omega_\eta^1 : \rho^{4D}\eta^{-3C-D}2^{-J} \leq \fH^1(\ell_\xi^- \cap \Omega_\eta^1) < (2\rho)^{4D}\eta^{-3C-D}2^{-J}\},
\end{align*}
where $D$ is a constant to be chosen.  Lemma \ref{first step} and the near-orthogonality of $\ell_\xi^+$ and $\ell_{\xi'}^-$ imply that $\fH^1(\ell_\xi^- \cap \Omega_\eta^1) \lesssim \eta^{-3C}2^{-J}$ for every $\xi \in \Omega_\eta^1$.  Thus, each $\xi \in \Omega_\eta^1$ belongs to a unique $\Omega_{\eta,\rho}^2$.

\begin{lemma}\label{second step}
For every $0 < \eta \leq \varepsilon$ and $0 < \rho \lesssim \eta^{1/4}$, we have $\|\fE_0\one_{\Omega'}\|_q \lesssim \rho^2|\Omega|^{1-2/q}$ for every measurable set $\Omega' \subseteq \Omega_{\eta,\rho}^2$.
\end{lemma}
\begin{proof}
If $\rho^{4D}\eta^{-3C-D} \geq \rho^{2D}$, then by Lemma \ref{first step}, we have
\begin{align*}
\|\fE_0\one_{\Omega'}\|_q \lesssim \eta^2|\Omega|^{1-\frac{2}{q}} \leq \rho^\frac{4D}{3C+D}|\Omega|^{1-\frac{2}{q}} \lesssim \rho^2|\Omega|^{1-\frac{2}{q}}
\end{align*}
for $D$ chosen sufficiently large.  Thus, we may assume that $\rho^{4D}\eta^{-3C-D} \leq \rho^{2D}$.  Given $\theta \in \Theta_{j,k}$, the set $\Omega' \cap 10\theta$ has $\pi^+$- and $\pi^-$-fibers of length at most $\min\{2^{-K},2^{-k}\}$ and $\min\{\rho^{2D}2^{-J},2^{-j}\}$, respectively, and the images of $\Omega' \cap 10\theta$ under $\pi^+$ and $\pi^-$ have measure at most $\min\{2^{-J},2^{-j}\}$ and $2^{-k}$, respectively.  Thus, by the coarea formula,
\begin{align}\label{fubini}
|\Omega' \cap 10\theta| \lesssim \min\{2^{-J-K},2^{-j-K},2^{-j-k},\rho^{2D}2^{-J-k}\}.
\end{align}
We define
\begin{align*}
R_1 &:= \{(j,k) : J- D\log\rho^{-1} \geq j,~K \geq k\} \cup \{(j,k) : J \geq j,~K-D\log\rho^{-1} \geq k\},\\
R_2 &:= \{(j,k) : j \geq J+D\log\rho^{-1},~K \geq k\} \cup \{(j,k) : j \geq J,~K-D\log\rho^{-1} \geq k\},\\
R_3 &:= \{(j,k) : j \geq J+D\log\rho^{-1},~k \geq K\} \cup \{(j,k) : j \geq J,~k \geq K+D\log\rho^{-1}\},\\
R_4 &:= \{(j,k) : J + D\log\rho^{-1} \geq j,~k+D\log\rho^{-1} \geq K\}.
\end{align*}
Each $(j,k)$ belongs to some $R_i$, so by \eqref{first est'} and \eqref{fubini}, we have
\begin{align*}
\|\fE_0\one_{\Omega'}\|_{q}^2 \lesssim \sum_{(j,k) \in R_1}2^{(j+k)(\frac{4}{q}-1)}2^{-(J+K)(1-\frac{2}{q})}|\Omega|^\frac{2}{q} + \sum_{(j,k) \in R_2}2^{(j+k)(\frac{4}{q}-1)}2^{-(j+K)(1-\frac{2}{q})}|\Omega|^\frac{2}{q}\\
+ \sum_{(j,k) \in R_3}2^{(j+k)(\frac{4}{q}-1)}2^{-(j+k)(1-\frac{2}{q})}|\Omega|^\frac{2}{q} + \sum_{(j,k) \in R_4}2^{(j+k)(\frac{4}{q}-1)}\rho^{2D(1-\frac{2}{q})}2^{-(J+k)(1-\frac{2}{q})}|\Omega|^\frac{2}{q}.
\end{align*}
Summing these geometric series leads to the bound $\|\fE_0\one_{\Omega'}\|_{q} \lesssim \rho^{D'}|\Omega|^{1-2/q}$, where $D'$ is a constant determined by $D$; increasing $D$ if necessary, we can make $D' \geq 2$.
\end{proof}

\emph{Step 3.} The final step of our decomposition is the same as the first, but with $\pi^-$ in place of $\pi^+$.  Indeed, each $\Omega_{\eta,\rho}^2$ has $\pi^-$-fibers of (essentially) constant length $\rho^{4D}\eta^{-3C-D}2^{-J}$.  For $0< \eta \leq \varepsilon$ and $0 < \rho \lesssim \eta^{1/4}$, let $S_{\eta,\rho} := \pi^-(\Omega_{\eta,\rho}^2)$.  Let $K_{\eta,\rho}$ be an integer such that $|S_{\eta,\rho}| \sim 2^{-K_{\eta,\rho}}$.  Let $\xi_1$ be a Lebesgue point of $S_{\eta,\rho}$ and $0 < \delta \leq \rho$ a dyadic number.  Define $I_{\eta,\rho,\delta}(\xi_1)$ to be the maximal dyadic interval $I$ such that $\xi_1 \in I$ and $|I \cap S_{\eta,\rho}|\geq \delta^C|I|$; as before, the Lebesgue differentiation theorem guarantees such an interval exists.  Let $T_{\eta,\rho,\delta} := \{\xi_1 \in S : |I_{\eta,\rho,\delta}(\xi_1)| \geq \delta^C2^{-K_{\eta,\rho}}\}$, and set $S_{\eta,\rho,\rho} := T_{\eta,\rho,\rho}$ and $S_{\eta,\rho,\delta} := T_{\eta,\rho,\delta} \setminus T_{\eta,\rho,2\delta}$ for $\delta < \rho$.  Finally, we let $\Omega_{\eta,\rho,\delta}^3 := \Omega_{\eta,\rho}^2 \cap (\pi^-)^{-1}(S_{\eta,\rho,\delta})$.

\begin{lemma}\label{third step}
For every $0 < \eta \leq \varepsilon$ and $0 < \delta \leq \rho \lesssim \eta^{1/4}$, the set $\Omega_{\eta,\rho,\delta}^3$ is contained in a union of $O(\delta^{-15C-4D})$ tiles in $\Theta_{J,K}$, and for every measurable set $\Omega' \subseteq \Omega_{\eta,\rho,\delta}^3$, we have
\begin{align*}
\|\fE_0\one_{\Omega'}\|_q \lesssim \delta^2|\Omega|^{1-\frac{2}{q}}.
\end{align*}
\end{lemma}
\begin{proof}
By an argument similar to the proof of Lemma \ref{first step}, one can show that
\begin{align*}
\|\fE_0 \one_{\Omega'}\|_q \lesssim \delta^2|\Omega_{\eta,\rho}^2|^{1-2/q} \leq \delta^2|\Omega|^{1-2/q}.
\end{align*}
Likewise, one sees that $\Omega_{\eta,\rho,\delta}$ is contained in a union of $O(\delta^{-3C})$ tiles in $\Theta_{n,K_{\eta,\rho}}$.  Since $\Omega_{\eta,\rho}^2$ has $\pi^-$-fibers of length at least $\rho^{4D}\eta^{-3C-D}2^{-J}$ and volume at most $2^{-J-K}$, we must have $2^{-K_{\eta,\rho}} \lesssim \rho^{-4D}2^{-K}$.  Thus, $\Omega_{\eta,\rho,\delta}^3$ is contained in $O(\delta^{-3C-4D})$ tiles in $\Theta_{n,K}$. By Lemma \ref{first step} and the fact that $\rho \lesssim \eta^{1/4}$, we also know that $\Omega_{\eta,\rho,\delta}^3$ is contained in $O(\delta^{-12C})$ tiles in $\Theta_{J,n}$.  The intersection of a tile in $\Theta_{J,n}$ and a tile in $\Theta_{n,K}$ is a tile in $\Theta_{J,K}$.
\end{proof}

We are now ready to complete the proof of Proposition \ref{decomposition}.  We set
\begin{align*}
\Omega_\delta := \bigcup_{\delta \leq \rho \lesssim \varepsilon^{1/4}}\bigcup_{\rho^4 \lesssim \eta \leq \varepsilon}\Omega_{\eta,\rho,\delta}^3,
\end{align*}
so that $\Omega = \bigcup_{0 < \delta \lesssim \varepsilon^{1/4}}\Omega_\delta$.  Since for fixed $\delta$ there are $O((\log\delta^{-1})^2)$ sets $\Omega_{\eta,\rho,\delta}^3$, properties (1) and (2) in the proposition follow from Lemma \ref{third step}.
\end{proof}

Now, fix some $\Omega \subseteq U$, and for each $K$, let $J(K)$ be an integer such that $|\Omega(K)^+| \sim 2^{-J(K)-K}$.  For each dyadic number $\varepsilon$, let $\fK(\varepsilon)$ denote the collection of all integers $K \geq 0$ for which $\varepsilon$ is the smallest dyadic number satisfying $\|\fE_0\one_{\Omega'}\|_q \lesssim \varepsilon^2|\Omega(K)^+|^{1-2/q}$ for every measurable set $\Omega' \subseteq \Omega(K)^+$.  For each $K \in \fK(\varepsilon)$, Proposition \ref{decomposition} produces a decomposition $\Omega(K)^+ = \bigcup_{0 < \delta \lesssim \varepsilon^{1/4}}\Omega(K)_\delta^+$ such that for each $\delta$, we have $\Omega(K)_\delta^+ \subseteq \bigcup_{\theta \in \Theta(K)_\delta}\theta$ for some $\Theta(K)_\delta \subseteq \Theta_{J(K),K}$ with $\#\Theta(K)_\delta \lesssim \delta^{-C_0}$.

\begin{lemma}\label{lem6}
For every $0 < \delta \lesssim \varepsilon^{1/4}$, we have
\begin{align*}
\bigg\|\sum_{K \in \fK(\varepsilon)}\fE_0\one_{\Omega(K)_\delta^+}\bigg\|_q^q \lesssim (\log \delta^{-1})^q\sum_{K \in \fK(\varepsilon)}\|\fE_0\one_{\Omega(K)_\delta^+}\|_q^q + \delta|\Omega|^{q-2}.
\end{align*}
\end{lemma}
\begin{proof}
Let $A$ be a constant to be chosen later, and divide $\fK(\varepsilon)$ into $O(\log \delta^{-1})$ subsets $\fK$ such that each is $A\log\delta^{-1}$-separated.  It suffices to prove that
\begin{align*}
\bigg\|\sum_{K \in \fK}\fE_0\one_{\Omega(K)_\delta^+}\bigg\|_q^q \lesssim \sum_{K \in \fK}\|\fE_0\one_{\Omega(K)_\delta^+}\|_q^q + \delta^2|\Omega|^{q-2}
\end{align*}
for each $\fK$.  We recall that $q < 4$. Thus,
\begin{align}\label{expansion}
\notag\bigg\|\sum_{K \in \fK}\fE_0\one_{\Omega(K)_\delta^+}\bigg\|_q^q &= \int\bigg\vert\sum_{\bK \in \fK^4}\prod_{i=1}^4 \fE_0\one_{\Omega(K_i)_\delta^+}\bigg\vert^\frac{q}{4}\\
&\lesssim \sum_{K \in \fK}\|\fE_0\one_{\Omega(K)_\delta^+}\|_q^q + \sum_{\bK \in \fK^4\setminus D(\fK^4)}\bigg\|\prod_{i=1}^4\fE_0\one_{\Omega(K_i)_\delta^+}\bigg\|_\frac{q}{4}^\frac{q}{4},
\end{align}
where $D(\fK^4) := \{\bK \in \fK^4 : K_1 = K_2 = K_3 = K_4\}$.  To control the latter sum, we have the following lemma.

\begin{lemma}
For all $K,K' \in \fK$, we have
\begin{align*}
\|\fE_0(\one_{\Omega(K)_\delta^+})\fE_0(\one_{\Omega(K')_\delta^+})\|_{q/2} \lesssim 2^{-c_0|K-K'|}\max\{|\Omega(K)^+|,|\Omega(K')^+|\}^{2-\frac{4}{q}}
\end{align*}
for some constant $c_0>0$.
\end{lemma}
\begin{proof}
Set $\tilde{\Omega} := \Omega(K)_\delta^+$, $\tilde{\Omega}' := \Omega(K')_\delta^+$, $J := J(K)$, and $J' := J(K')$.  By the Cauchy--Schwarz inequality and Proposition \ref{const fiber}, we have
\begin{align*}
\|\fE_0(\one_{\tilde{\Omega}})\fE_0(\one_{\tilde{\Omega}'})\|_{q/2} \lesssim |\tilde{\Omega}|^{1-\frac{2}{q}}|\tilde{\Omega}'|^{1-\frac{2}{q}}.
\end{align*}
If either (i) $K = K'$, (ii) $J = J'$, (iii) $J < J'$ and $K < K'$, or (iv) $J > J'$ and $K > K'$, then
\begin{align*}
|\tilde{\Omega}|^{1-\frac{2}{q}}|\tilde{\Omega}'|^{1-\frac{2}{q}} \lesssim 2^{-(1-\frac{2}{q})|K-K'|}\max\{|\tilde{\Omega}|,|\tilde{\Omega}'|\}^{2-\frac{4}{q}}.
\end{align*}
Thus, by symmetry, we may assume that $K < K'$ and $J > J'$.  By the bound $\#(\Theta(K)_\delta \times \Theta(K')_\delta) \lesssim \delta^{-2C_0}$ and the separation condition on $\fK$ (with $A$ sufficiently large), it suffices to prove that
\begin{align*}
\|\fE_0(\one_{\tilde{\Omega} \cap \theta})\fE_0(\one_{\tilde{\Omega}' \cap \sigma})\|_{q/2} \lesssim 2^{-c|K-K'|}\max\{|\tilde{\Omega}|,|\tilde{\Omega}'|\}^{2-\frac{4}{q}}
\end{align*}
for all $\theta \in \Theta(K)_\delta$, $\sigma \in \Theta(K')_\delta$, and some constant $c>0$.

Fix two such tiles $\theta,\sigma$, and set $\tau := \Phi^{-1}(\theta)$ and $\kappa := \Phi^{-1}(\sigma)$.  Thus, $\tau$ and $\kappa$ are dyadic rectangles of dimensions $2^{-J} \times 2^{-K}$ and $2^{-J'}\times 2^{-K'}$, respectively.  We note that our assumptions on $J,J',K,K'$ imply that $\tau$ is taller than $\kappa$ and $\kappa$ wider than $\tau$.  By translation, we may assume that the $\zeta_2$- and $\zeta_1$-axes intersect the centers of $\tau$ and $\kappa$, respectively.  Define
\begin{align*}
\tau_k := \begin{cases} \tau \cap \{\zeta : |\zeta_2| \sim 2^{-k}\}, &k < K',\\ \tau \cap \{\zeta : |\zeta_2| \lesssim 2^{-K'}\}, &k=K' \end{cases} \quad\quad\text{and}\quad\quad \kappa_j := \begin{cases} \kappa \cap \{\zeta : |\zeta_1| \sim 2^{-j}\}, &j < J,\\ \kappa \cap \{\zeta : |\zeta_1| \lesssim 2^{-J}\}, &j=J \end{cases},
\end{align*}
as well as $\theta_k := \Phi(\tau_k)$ and $\sigma_j := \Phi(\kappa_j)$.  Thus,
\begin{align*}
\theta = \bigcup_{k=0}^{K'}\theta_k \quad\quad\text{and}\quad\quad \sigma = \bigcup_{j=0}^J \sigma_j,
\end{align*}
so that by the triangle inequality,
\begin{align*}
\|\fE_0(\one_{\tilde{\Omega} \cap \theta})\fE_0(\one_{\tilde{\Omega}' \cap \sigma})\|_{q/2} \leq \sum_{k=0}^{K'}\sum_{j=0}^J\|\fE_0(\one_{\tilde{\Omega} \cap \theta_k})\fE_0(\one_{\tilde{\Omega}' \cap \sigma_j})\|_{q/2}.
\end{align*}

We first sum the terms with $k = K'$.  By the Cauchy--Schwarz inequality, Proposition \ref{const fiber}, and the fact that $|\det \nabla \Phi| \sim 1$, we have
\begin{align*}
\sum_{j=0}^J\|\fE_0(\one_{\tilde{\Omega} \cap \theta_{K'}})\fE_0(\one_{\tilde{\Omega}' \cap \sigma_j})\|_{q/2} \lesssim \sum_{j=0}^J |\theta_{K'}|^{1-\frac{2}{q}}|\sigma_j|^{1-\frac{2}{q}} \sim \sum_{j=0}^J|\tau_{K'}|^{1-\frac{2}{q}}|\kappa_j|^{1-\frac{2}{q}}.
\end{align*}
Since $\kappa$ has width $2^{-J'}$, there are at most two nonempty $\kappa_j$ with $j \leq J'$.  This fact and the bound
\begin{align}\label{small tile}
|\kappa_j| \leq \min\{2^{-(j-J')},1\}|\kappa|
\end{align}
imply that $\sum_{j=0}^J |\kappa_j|^{1-2/q} \lesssim |\kappa|^{1-2/q}$.  Since $|\tau_{K'}| \lesssim 2^{-(K'-K)}|\tau|$, $|\tau| \sim |\tilde{\Omega}|$, and $|\kappa| \sim |\tilde{\Omega}'|$, we altogether have
\begin{align*}
\sum_{j=0}^J\|\fE_0(\one_{\tilde{\Omega} \cap \theta_{K'}})\fE_0(\one_{\tilde{\Omega}' \cap \sigma_j})\|_{q/2} \lesssim 2^{-(K'-K)(1-\frac{2}{q})}|\tilde{\Omega}|^{1-\frac{2}{q}}|\tilde{\Omega}'|^{1-\frac{2}{q}},
\end{align*}
which is acceptable.  A similar argument shows that
\begin{align*}
\sum_{k=0}^{K'}\|\fE_0(\one_{\tilde{\Omega} \cap \theta_{k}})\fE_0(\one_{\tilde{\Omega}' \cap \sigma_J})\|_{q/2} &\lesssim 2^{-(J-J')(1-\frac{2}{q})}|\tilde{\Omega}|^{1-\frac{2}{q}}|\tilde{\Omega}'|^{1-\frac{2}{q}}\\
&\sim 2^{-(K'-K)(1-\frac{2}{q})}|\tilde{\Omega}|^{2-\frac{4}{q}}.
\end{align*}

We now consider the terms with $k < K'$ and $j < J$.  In this case, $\tau_k$ is contained in a union of four dyadic rectangles of dimensions $2^{-J} \times 2^{-\max\{K,k\}}$, and $\kappa_j$ is contained in a union of four dyadic rectangles of dimensions $2^{-\max\{J',j\}}\times 2^{-K'}$.  Moreover, these rectangles are separated by a distance of (at least) $2^{-k}$ and $2^{-j}$ in the vertical and horizontal directions, respectively.  Thus, we can apply Theorem \ref{T:bilinear tiles} to $\theta_k$ and $\sigma_j$ to get
\begin{align*}
\|\fE_0(\one_{\tilde{\Omega}\cap\theta_k})\fE_0(\one_{\tilde{\Omega}'\cap\sigma_j})\|_{q/2} \lesssim 2^{(j+k)(\frac{4}{q}-1)}|\tilde{\Omega}\cap\theta_k|^\frac{1}{2}|\tilde{\Omega}'\cap\sigma_j|^\frac{1}{2}.
\end{align*}
Using \eqref{small tile} and the analogous bound for $|\tau_k|$, we now get
\begin{align*}
\sum_{k=0}^{K'-1}\sum_{j=0}^{J-1}\|\fE_0(\one_{\tilde{\Omega}\cap\theta_k})\fE_0(\one_{\tilde{\Omega}'\cap\sigma_j})\|_{q/2} &\lesssim 2^{(J'+K)(\frac{4}{q}-1)}|\theta|^\frac{1}{2}|\sigma|^\frac{1}{2}\\
&\sim 2^{(J'-J+K-K')(\frac{2}{q}-\frac{1}{2})}|\tilde{\Omega}|^{1-\frac{2}{q}}|\tilde{\Omega}'|^{1-\frac{2}{q}}.
\end{align*}
By the relations $K < K'$ and $J > J'$ and the fact that $q < 4$, the lemma is proved.
\end{proof}

Returning to the proof of Lemma \ref{lem6}, we consider the second sum in \eqref{expansion}.  Given $\bK \in \fK^4 \setminus D(\fK^4)$, let $p(\bK) = (p_i(\bK))_{i=1}^4$ be a permutation of $\bK$ such that $|\Omega(p_1(\bK))^+|$ is maximal among $|\Omega(K_i)^+|$, $1 \leq i \leq 4$, and such that $|K_i - K_j| \leq 2|p_1(\bK)-p_2(\bK)|$ for all $1 \leq i,j \leq 4$.  Then by the Cauchy--Schwarz inequality, Lemma \ref{lem6}, the separation condition on $\fK$, the fact that $q > 3$, and choosing $A$ sufficiently large, we get
\begin{align*}
\sum_{\bK \in \fK^4 \setminus D(\fK^4)}\bigg\|\prod_{i=1}^4\fE_0\one_{\Omega(K_i)_\delta^+}\bigg\|_\frac{q}{4}^\frac{q}{4} &\lesssim \sum_{\substack{\bK \in \fK^4 \setminus D(\fK^4)\\ \bK = p(\bK)}} 2^{-c_0|p_1(\bK)-p_2(\bK)|}|\Omega(p_1(\bK))^+|^{q-2}\\
&\lesssim \sum_{K_1 \in \fK}\sum_{K_2 \in \fK}|K_1-K_2|^2 2^{-c_0|K_1-K_2|}|\Omega(K_1)^+|^{q-2}\\
&\lesssim \delta^\frac{c_0A}{2}\sum_{K_1 \in \fK}|\Omega(K_1)^+|^{q-2} \lesssim \delta^2|\Omega|^{q-2}.
\end{align*}
This concludes the proof of Lemma \ref{lem6}.
\end{proof}


\begin{proof}[Proof of Theorem  \ref{linear est zero region}]
By interpolation, it suffices to prove the analogous restricted strong type estimate.  Let $\Omega \subseteq U$ be a measurable set.  We have the decomposition
\begin{align*}
\Omega = \bigcup_{0 < \varepsilon \lesssim 1} \bigcup_{0 < \delta \lesssim \varepsilon^{1/4}} \bigcup_{K \in \fK(\varepsilon)} \Omega(K)_\delta^+.
\end{align*}
Thus, by the triangle inequality, Lemma \ref{lem6}, Proposition \ref{decomposition}, and the fact that $q > 3$, we obtain
\begin{align*}
\|\fE_0\one_\Omega\|_{q} &\leq \sum_{0 < \varepsilon \lesssim 1}\sum_{0 < \delta \lesssim \varepsilon^{1/4}}\bigg\|\sum_{K \in \fK(\varepsilon)}\fE_0\one_{\Omega(K)_\delta^+}\bigg\|_{q}\\
&\lesssim \sum_{0 < \varepsilon \lesssim 1}\sum_{0 < \delta \lesssim \varepsilon^{1/4}}\bigg((\log\delta^{-1})^{q}\sum_{K \in \fK(\varepsilon)}\|\fE_0\one_{\Omega(K)_\delta^+}\|_{q}^{q} + \delta|\Omega|^{q-2}\bigg)^\frac{1}{q}\\
&\lesssim \Bigg[\sum_{0 < \varepsilon \lesssim 1}\sum_{0 < \delta \lesssim \varepsilon^{1/4}} (\log\delta^{-1})\delta\bigg(\sum_{K \in \fK(\varepsilon)}|\Omega(K)^+|^{q-2}\bigg)^\frac{1}{q}\Bigg] + |\Omega|^{1-\frac{2}{q}}\\
&\lesssim \Bigg[\sum_{0 < \varepsilon \lesssim 1}\sum_{0 < \delta \lesssim \varepsilon^{1/4}} (\log\delta^{-1})\delta |\Omega|^{1-\frac{2}{q}}\Bigg] + |\Omega|^{1-\frac{2}{q}}\lesssim |\Omega|^{1-\frac{2}{q}}.
\end{align*}
\end{proof}

\section{Proof of Theorem \ref{T:bilinear tiles}} \label{S:bilinear tiles}
In this section, we prove Theorem \ref{T:bilinear tiles}, the bilinear restriction estimate for related tiles.  As mentioned above, we proceed by rescaling a result of Lee \cite{Lee neg}.

We begin by defining some notation.  The basic symmetries of the hyperbolic hyperboloid are the Lorentz transformations, which, given the parametrization \eqref{param hyp}, are the linear maps on $\R \times \R^2$ that preserve the quadratic form $(\tau,\xi) \mapsto \tau^2-\xi_1^2 + \xi_2^2$.  The Lorentz-invariant measure $\d\sigma$ on $\Sigma$ takes the form
\begin{align*}
 \int_\Sigma g \,\d\sigma := \int_U g(\phi(\xi),\xi)\frac{\d\xi}{\phi(\xi)},  
\end{align*}
where $\phi(\xi) = \sqrt{1+\xi_1^2-\xi_2^2}$ as before.  If $L$ is a Lorentz transformation and $\supp\,g  \subseteq \Sigma$ and $L^{-1}(\supp\,g) \subseteq \Sigma$, then
\begin{align*}
    \int_\Sigma (g \circ L) \d\sigma = \int_\Sigma g \,\d\sigma.
\end{align*}
Let $\Omega := \{\xi \in \R^2 :1+\xi_1^2-\xi_2^2 \geq 0\}$.  Given a Lorentz transformation $L$ and $\xi \in \Omega$, let
\begin{align*}
\overline{L}(\xi) := \pi(L(\phi(\xi),\xi)),
\end{align*}
where $\pi(\tau',\xi') := \xi'$ is the projection to the spatial coordinates. If $\xi \in \Omega$ and $e_1 \cdot L(\phi(\xi),\xi) \geq 0$, where $e_1 = (1,0,0)$ denotes the first standard basis vector, then $\overline{ML}(\xi) = \overline{M}(\overline{L}(\xi))$ for any other Lorentz transformation $M$.  In particular, if $E \subseteq \Omega$ and $e_1\cdot L(\phi(\xi),\xi) \geq 0$ for all $\xi \in E$, then $\overline{L}$ is invertible on $E$ with $\overline{L}^{-1}(\zeta) = \overline{L^{-1}}(\zeta)$ for $\zeta \in \overline{L}(E)$.

We now turn to the proof of Theorem \ref{T:bilinear tiles}. We may assume that $j \geq k$.  Fix $\theta_1 \sim \theta_2 \in \Theta_{j,k}$ and $c \in \theta_1 \cup \theta_2$.  Arguing as in the proof of Lemma \ref{tiles}, there exists an $O(2^{-j}) \times O(2^{-k})$ rectangle $S$ centered at $c$ with major axis $\ell_c^+$ such that $\theta_1 \cup \theta_2 \subseteq S$.  We define three Lorentz transformations,
\begin{align*}
R(\tau,\xi) &:= (-\omega_2\xi_2 + \omega_1\tau, \xi_1, \omega_1\xi_2+\omega_2\tau), \quad\quad\quad \omega := \bigg(\frac{\phi(c)}{\sqrt{1+c_1^2}},-\frac{c_2}{\sqrt{1+c_1^2}}\bigg) \in \mathbb{S}^1,\\
B(\tau,\xi) &:= \Big(-c_1\xi_1 + \sqrt{1+c_1^2}\tau, \sqrt{1+c_1^2}\xi_1-c_1\tau,\xi_2\Big),\\
D(\tau,\xi) &:= \frac{1}{2}\big(2\tau, \big(2^\frac{k-j}{2}+2^\frac{j-k}{2}\big)\xi_1 + \big(2^\frac{k-j}{2}-2^\frac{j-k}{2}\big)\xi_2, \big(2^\frac{k-j}{2}-2^\frac{j-k}{2}\big)\xi_1 + \big(2^\frac{k-j}{2}+2^\frac{j-k}{2}\big)\xi_2\big),
\end{align*}
which correspond to a spatial rotation, a boost, and a dilation, respectively. The composition $BR$ takes $(\phi(c),c)$ to $(1,0,0)$. We will show that $\overline{DBR}$ essentially takes $\theta_1,\theta_2$ to a pair of $O(2^{-\frac{j+k}{2}})$-squares near the origin, which we will then parabolically rescale to size $1$. After checking that the separation of $\theta_1$ and $\theta_2$ is respected by these rescalings, we will apply Lee's result \cite[Theorem 1.1]{Lee neg} to finish the proof.

We turn to the details.  Because Lorentz transformations preserve the hyperbolic hyperboloid and are linear, they must permute the rulings of the surface.  Thus, since $BR(\phi(c),c) = (1,0,0)$, we have either $\overline{BR}(\ell_c^+) = \ell_0^+$ or $\overline{BR}(\ell_c^+) = \ell_0^-$.  It is easy to check that in fact $\overline{BR}(\ell_c^+) = \ell_0^+ = \R(1,1)$ and that $\|\nabla \overline{BR}\| \lesssim 1$ near the origin.  Thus, by the definition of the rectangle $S$, it follows that $\overline{BR}(\theta_1 \cup \theta_2)$ is contained in an $O(2^{-j}) \times O(2^{-k})$ rectangle of slope $1$ centered at the origin. Since $D$ contracts by a factor of $O(2^\frac{k-j}{2})$ in the direction of $(1,1)$ and expands by a factor of $O(2^\frac{j-k}{2})$ in the orthogonal direction, $\overline{D}(\overline{BR}(\theta_1 \cup \theta_2))$ lies in a disc $V$ of radius $O(2^{-\frac{j+k}{2}})$ centered at $0$.  If $j$ and $k$ are sufficiently large (which we may assume), then $V \subseteq U$.  It is easy to check that $e_1 \cdot BR(\phi(\xi),\xi) \geq 0$ for all $\xi \in U$. Thus, setting $L = (DBR)^{-1}$, we have
\begin{align}\label{proj}
    \overline{L^{-1}}(\theta_1 \cup \theta_2) = \overline{D}(\overline{BR}(\theta_1 \cup \theta_2)) \subseteq V.
\end{align}

Let $Q_i := \overline{L}^{-1}(\theta_i) := \{\xi \in \Omega : \overline{L}(\xi) \in \theta_i\}$.  We claim that
\begin{align}\label{preimage}
Q_i = \overline{L^{-1}}(\theta_i).
\end{align}
Given a set $E \subseteq \Omega$, let $E^\pm := \{(\pm\phi(\xi),\xi) : \xi \in E\}$.  Then we have
\begin{align*}
    Q_i &= \{\xi \in \Omega : L(\phi(\xi),\xi) \in \theta_i^+ \cup \theta_i^-\}\\
    &= \{\xi \in \Omega : (\phi(\xi),\xi) \in L^{-1}(\theta_i^+) \cup (-L^{-1}((-\theta_i)^+))\}.
\end{align*}
It is easy to check that $e_1\cdot L^{-1}(\phi(\zeta),\zeta) > 0$ for all $\zeta \in U$.  Thus, since $-\theta_i \subseteq U$ and $\phi \geq 0$, we have $(\phi(\xi),\xi) \notin -L^{-1}((-\theta_i)^+)$ for every $\xi$.  Hence
\begin{align*}
    Q_i = \{\xi \in \Omega : (\phi(\xi),\xi) \in L^{-1}(\theta_i^+)\} = \overline{L^{-1}}(\theta_i),
\end{align*}
proving the claim.

Now, if $f$ is supported in $\theta_1 \cup \theta_2$, then by the Lorentz invariance of the measure $\d\sigma$,
\begin{align*}
\fE_0 f(t,x) = \fE_0 f_L(L^*(t,x)),
\end{align*}
where
\begin{align*}
    f_L(\xi) := \frac{f(\overline{L}(\xi))\phi(\overline{L}(\xi))}{\phi(\xi)}.
\end{align*}
We have $|f_L| \sim |f \circ \overline{L}|$ on $\supp f_L \subseteq Q_1 \cup Q_2 \subseteq V$.  Additionally, $e_1\cdot L(\phi(\xi),\xi) \geq 0$ for all $\xi \in V$, so we know that $\overline{L}$ is invertible on $Q_1 \cup Q_2$ with $\overline{L}^{-1}(\zeta) = \overline{L^{-1}}(\zeta)$ for $\zeta \in \overline{L}(Q_1 \cup Q_2)$.  A straightforward calculation shows that $|\det\nabla \overline{L}^{-1}(\zeta)| \lesssim 1$ on $\overline{L}(Q_1 \cup Q_2)$.  Combining these observations, we see that the estimate in Theorem \ref{T:bilinear tiles} is equivalent to
\begin{align}\label{equiv est}
\|\fE_0 f\fE_0 g\|_{q/2} \lesssim 2^{(j+k)(\frac{4}{q}-1)}\|f\|_2\|g\|_2
\end{align}
for all $f \in L^2(Q_1)$ and $g \in L^2(Q_2)$.

Now, by parabolic rescaling, we have
\begin{align*}
\fE_0 f(t,x) = 2^{-j-k}\fE_0^\psi [f(2^{-\frac{j+k}{2}}\cdot)](2^{-j-k}t, 2^{-\frac{j+k}{2}}x),
\end{align*}
for every $f$ supported in $Q_1 \cup Q_2 \subseteq V$, where $\fE_0^\psi$ is the extension operator associated to the phase $\psi(\xi) := 2^{j+k}\sqrt{1+2^{-j-k}(\xi_1^2-\xi_2^2)}$.  The estimate \eqref{equiv est} now follows from \cite[Theorem 1.1]{Lee neg}, provided the hypotheses of the latter are satisfied.  Let $\tilde{Q}_i := 2^\frac{j+k}{2}Q_i$.  We need to the check that
\begin{align*}
|\langle (\nabla^2\psi(\xi''))^{-1}(\nabla\psi(\xi)-\nabla\psi(\zeta)), \nabla\psi(\xi')-\nabla\psi(\zeta')\rangle| &\gtrsim 1,\\
|\langle (\nabla^2\psi(\zeta''))^{-1}(\nabla\psi(\xi)-\nabla\psi(\zeta)), \nabla\psi(\xi')-\nabla\psi(\zeta')\rangle| &\gtrsim 1,
\end{align*}
for all $\xi,\xi',\xi'' \in \tilde{Q}_1$ and $\zeta,\zeta',\zeta'' \in \tilde{Q}_2$.  Let $r(\xi) := (\xi_1,-\xi_2)$.  Simple calculations and the mean value theorem show that
\begin{align*}
\nabla\psi(\xi) &= r(\xi) + O(2^{-j-k}),\\
(\nabla^2\psi(\xi))^{-1} &= \left(\begin{array}{cc}1&0\\0&-1\end{array}\right) + O(2^{-j-k}),
\end{align*}
for all $|\xi| \leq 1$, and thus we only need to show that
\begin{align}\label{ex9}
|\langle \xi - \zeta, r(\xi'-\zeta')\rangle| \gtrsim 1
\end{align}
for all $\xi,\xi' \in \tilde{Q_1}$ and $\zeta,\zeta' \in \tilde{Q_2}$.

Let $p(\xi)$ and $q(\xi)$ denote the orthogonal projections of $\xi$ to the lines $\R(1,-1)$ and $\R(1,1)$, respectively.  That is,
\begin{align*}
p(\xi) &:= \frac{1}{2}(\xi_1-\xi_2,\xi_2-\xi_1)\\
q(\xi) &:= \frac{1}{2}(\xi_1+\xi_2,\xi_1+\xi_2).
\end{align*}
Assume for now that the following lemma holds:
\begin{lemma}\label{lem8}
For all $\xi \in Q_1$ and $\zeta \in Q_2$, we have
\begin{align*}
|p(\xi) - p(\zeta)| &\gtrsim 2^{-\frac{j+k}{2}},\\
|q(\xi) - q(\zeta)| &\gtrsim 2^{-\frac{j+k}{2}}.
\end{align*} 
\end{lemma}

We immediately see that $|p(\xi-\zeta)| \gtrsim 1$ and $|q(\xi-\zeta)| \gtrsim 1$ for all $\xi \in \tilde{Q}_1$ and $\zeta \in \tilde{Q}_2$.  We will use these bounds to prove \eqref{ex9}.  We have
\begin{align*}
\langle \xi - \zeta, r(\xi' - \zeta') \rangle &= \langle p(\xi-\zeta) + q(\xi-\zeta), p(r(\xi'-\zeta')) + q(r(\xi'-\zeta'))\rangle\\
&= \langle p(\xi-\zeta),p(r(\xi'-\zeta'))\rangle + \langle q(\xi-\zeta),q(r(\xi'-\zeta'))\rangle
\end{align*}
by orthogonality.  Using the relations $p \circ r = r \circ q$ and $q \circ r = r \circ p$ and the fact that $r$ is unitary, we thus have
\begin{align}\label{ex10}
\langle \xi- \zeta, r(\xi'-\zeta')\rangle = \langle r(p(\xi-\zeta)), q(\xi'-\zeta') \rangle + \langle q(\xi-\zeta), r(p(\xi'-\zeta'))\rangle.
\end{align}
Using the fact that the sets $q(\tilde{Q}_1)$ and $q(\tilde{Q}_2)$ (resp.~$p(\tilde{Q}_1)$ and $p(\tilde{Q}_2))$ are disjoint and each of them is connected, one sees that $q(\xi-\zeta)$ and $q(\xi'-\zeta')$ are parallel, as are $p(\xi-\zeta)$ and $p(\xi'-\zeta')$.  Since $r$ is unitary, $r(p(\xi-\zeta))$ and $r(p(\xi'-\zeta'))$ are also parallel.  Hence, both terms on the right-hand side of \eqref{ex10} have the same sign, and thus it suffices to bound one of them from below.  We have
\begin{align*}
|\langle r(p(\xi-\zeta)),q(\xi'-\zeta')\rangle| = |p(\xi-\zeta)||q(\xi'-\zeta')| \gtrsim 1.
\end{align*}

It remains to prove Lemma \ref{lem8}.
\begin{proof}[Proof of Lemma \ref{lem8}]
We know that $Q_i = \overline{L^{-1}}(\theta_i) = \overline{D}(\overline{BR}(\theta_i))$, by \eqref{preimage} and \eqref{proj}.  We claim, first, that
\begin{align}\label{initial bound}
\notag|\pi^+(\xi)-\pi^+(\zeta)| &\gtrsim 2^{-j},\\ |\pi^-(\xi)-\pi^-(\zeta)| &\gtrsim 2^{-k}
\end{align}
for all $\xi \in \overline{BR}(\theta_1)$ and $\zeta \in \overline{BR}(\theta_2)$, where $\pi^\pm$ are the projections used throughout the previous section (recall \eqref{eq:defpi}).  Using that $\|(BR)^{-1}\| \lesssim 1$, it is not difficult to show that $\overline{BR}^{-1}$ exists in a neighborhood of $0$ (of constant size) and is given by $\overline{BR}^{-1}(\eta) = \overline{(BR)^{-1}}(\eta)$.  We showed above that the sets $\overline{BR}(\theta_i)$ lie in a disc of radius $O(2^{-k})$ centered at $0$.  Fix $\xi \in \overline{BR}(\theta_1)$ and $\zeta \in \overline{BR}(\theta_2)$, and let $\xi' := \overline{BR}^{-1}(\xi)$, $\alpha := \overline{BR}^{-1}(\pi^+(\xi),0)$ and $\zeta' := \overline{BR}^{-1}(\zeta)$, $\beta := \overline{BR}^{-1}(\pi^+(\zeta),0)$. Since $(\pi^+(\xi),0) \in \ell_\xi^+$ by Lemma \ref{properties of maps}, it follows that $\alpha \in \ell_{\xi'}^+$.  Similarly, $\beta \in \ell_{\zeta'}^+$.  Thus, since $\pi^+$ is constant along lines of the form $\ell_\eta^+$, by Lemma \ref{properties of maps}, and $\theta_1 \sim \theta_2$, we have
\begin{align*}
    2^{-j} &\lesssim |\pi^+(\xi') - \pi^+(\zeta')|\\
    &= |\pi^+(\alpha)-\pi^+(\beta)|\\
    &\lesssim |\alpha - \beta|\\
    &= |\overline{BR}^{-1}(\pi^+(\xi),0) - \overline{BR}^{-1}(\pi^+(\zeta),0)|\\
    &\lesssim |\pi^+(\xi) - \pi^+(\zeta)|.
\end{align*}
A similar argument gives the second estimate in \eqref{initial bound}.

We will now use \eqref{initial bound} to prove the lemma.  We only prove the first estimate; the second one follows by a similar argument.  Fix $\xi \in Q_1$ and $\zeta \in Q_2$, and let $\alpha = \overline{D}^{-1}(p(\xi)) = 2^\frac{k-j}{2}p(\xi)$ and $\beta = \overline{D}^{-1}(p(\zeta)) = 2^\frac{k-j}{2}p(\zeta)$.  It suffices to show that $|\alpha - \beta| \gtrsim 2^{-j}$.  Let $\xi'$ be the intersection of the lines $\ell_\xi^+$ and $\R(1,-1)$.  The points $A := \xi$, $B := p(\xi)$, and $C := \xi'$ form a right triangle with hypotenuse $AC$.  One easily checks that $\angle(\ell_\xi^+,\R(1,1)) \lesssim |\xi|$.  Thus, $\angle CAB \lesssim 2^{-\frac{j+k}{2}}$ by the fact that $\xi \in Q_1 \subseteq V$.  We also have $|AC| \lesssim 2^{-\frac{j+k}{2}}$, and thus $|p(\xi) - \xi'| = |BC| \lesssim 2^{-j-k}$.  Let $\alpha' = \overline{D}^{-1}(\xi') = 2^\frac{k-j}{2}\xi'$.  Then $|\alpha - \alpha'| = 2^\frac{k-j}{2}|p(\xi)-\xi'| \lesssim 2^{-\frac{3j+k}{2}}$.  Because $\xi' \in \ell_\xi^+$, we have $\ell_{\xi'}^+ = \ell_\xi^+$ (see the proof of Lemma \ref{properties of maps}) and thus $\alpha' \in \overline{D}^{-1}(\ell_\xi^+) = \ell_{\overline{D}^{-1}(\xi)}^+$.  Thus, since $\pi^+$ is constant along $\ell_{\overline{D}^{-1}(\xi)}^+$, we have
\begin{align*}
|\pi^+(\alpha) - \pi^+(\overline{D}^{-1}(\xi))| = |\pi^+(\alpha) - \pi^+(\alpha')| \lesssim |\alpha - \alpha'| \lesssim 2^{-\frac{3j+k}{2}},
\end{align*}
and by a similar argument,
\begin{align*}
|\pi^+(\beta) - \pi^+(\overline{D}^{-1}(\zeta))| \lesssim 2^{-\frac{3j+k}{2}}.
\end{align*}
Since $\overline{D}^{-1}(\xi) \in \overline{B}\overline{R}(\theta_1)$ and $\overline{D}^{-1}(\zeta) \in \overline{B}\overline{R}(\theta_2)$, the first estimate in \eqref{initial bound} now gives the bound
\begin{align*}
2^{-j} \lesssim |\pi^+(\overline{D}^{-1}(\xi)) - \pi^+(\overline{D}^{-1}(\zeta))| = |\pi^+(\alpha) - \pi^+(\beta)| + O(2^{-\frac{3j+k}{2}}),
\end{align*}
and consequently
\begin{align*}
|\alpha - \beta| \gtrsim |\pi^+(\alpha) - \pi^+(\beta)| \gtrsim 2^{-j},
\end{align*}
which is what we needed to show.
\end{proof}

\begin{remark}\label{scaling line}
As an application, we pause to explain how the bilinear theory for $\fE_0$ can be used to obtain further (conditional) linear estimates for $\fE_0$ on the parabolic scaling line $p = (q/2)'$.  Similar to the case of the hyperbolic paraboloid (see \cite{BS neg}), the proof of Theorem \ref{linear est zero region} can be adjusted to give the following conditional bilinear-to-linear result:  Given $3 < q_0 < 4$, if there exists some $p_0 < (\frac{q_0}{2})'$ such that 
\[\|\fE_0 f\fE_0 g\|_{q_0/2} \lesssim 2^{(j+k)(\frac{4}{q_0}+\frac{2}{p_0}-2)}\|f\|_{p_0}\|g\|_{p_0},\] 
for all functions $f$, $g$ supported in related tiles in $\Theta_{j,k}$, then $\fE_0$ is bounded from $L^{(q/2)'}$ to $L^q$ for all $q > q_0$.  In \cite{BenHypHyp}, the first author showed the following: If $q_0>3.25$, $p_0 > (\frac{q_0}{2})'$, and $0 < r \leq 1$, then $\|\fE_0^r f\|_{q_0} \lesssim \|f\|_{p_0}$ uniformly in $r$, where
\begin{align*}
    \fE_0^r f(t,x) := \int_U e^{i(t,x)\cdot(r^{-2}\sqrt{1+r^2(\xi_1^2-\xi_2^2)},\xi)}f(\xi)\d\xi.
\end{align*}
Using this result, the Cauchy--Schwarz inequality, a parabolic rescaling argument (utilizing the uniformity in $r$), and interpolation with Theorem \ref{T:bilinear tiles}, one can show that the hypothesis of the conditional version of Theorem \ref{linear est zero region} holds for each $q_0 > 3.25$.  We conclude that $\fE_0$ is bounded from $L^{(q/2)'}$ to $L^q$ for every $q > 3.25$.
\end{remark}

\section{Bilinear adjoint restriction on annuli}\label{S:bilinear conic}

In the next two sections we establish bounds for the extension operator associated to dyadic annuli in our hyperboloid.  By invariance under cylindrical rotations and the triangle inequality, it suffices to consider subsets of these annuli with some angular restriction, and we abuse notation (relative to the introduction) by defining
$$
\Gamma_N:=\{(\tau,\xi) \in \Gamma : |\xi| \sim 2^N,~ |\tfrac{\xi}{|\xi|}-e_1| < 0.001\},
$$
where $e_1$ denotes the usual first coordinate vector.  We will use the notation $f_N$ to denote a function supported on $\Gamma_N$.

The focus of this section will be on establishing bounds in the bilinear range, where our results are unconditional and our deduction is more straightforward.  We will then turn to the conditional result in the next section, the proof of which will use some of the lemmas from this section.  

\begin{proposition}\label{P:annuli bilin range}
Let $(\tfrac q2)' \leq p \leq q$ and $4 > q > \tfrac{10}3$.  Then 
$$
\|\scriptE f_N\|_q \lesssim \|f_N\|_p,
$$
for all functions $f_N$ supported on $\Gamma_N$.  
\end{proposition}

The remainder of this section will be devoted to the proof of Proposition~\ref{P:annuli bilin range}.  

We will work on sectors of varying width contained in the $\Gamma_N$.  Let $C \leq k \leq N$.  By an {\it $(N,k)$-sector}, we mean a set of the form
$$
\Gamma_{N,k}^\omega:=\{(\tau,\xi) \in \Gamma_N : |\tfrac{\xi}{|\xi|}-\omega| < 2^{-k}\},
$$
with $\omega \in \mathbb{S}^1$; we refer to $2^{-k}$ as the angular width of the sector.  

We begin by establishing bounds on the thinnest sectors.  

\begin{lemma}\label{L:zero to thin sector}
For any $p,q$, validity of $\scriptR_0^*(p \to q)$ implies that
\begin{equation} \label{E:zero to thin sector}
\|\scriptE f_{N,N}^\omega\|_q \lesssim \|f_{N,N}^\omega\|_p,
\end{equation}
for every function $f_{N,N}^\omega$ supported in an $(N,N)$-sector, $N \geq 1$.  In particular, \eqref{E:zero to thin sector} holds for all $q \geq 2p'$ when $q>\tfrac{10}3$.  
\end{lemma}

\begin{proof}
We recall the definition \eqref{E:Lorentz} of the Lorentz boost $L_\nu$ and the Lorentz invariance of our measure. The deduction claimed in the lemma follows from the observation that if $\omega \in \mathbb{S}^1$ and $N \geq 1$, $L_{2^N \omega}$ maps $\Gamma_{N,N}^\omega$ into $\Gamma_0$.  
\end{proof}

Now we turn to the deduction of bounds on the $\Gamma_N$ from those on the $\Gamma_{N,N}^\omega$, for which we adapt the bilinear theory for the cone.  

For $k < N$, we say that two $(N,k)$-sectors, $\Gamma_{N,k}^{\omega}$ and $\Gamma_{N,k}^{\omega'}$ are related, $\Gamma_{N,k}^{\omega} \sim \Gamma_{N,k}^{\omega'}$, when $2^{-k+4}\leq  |\omega - \omega'| \leq 2^{-k+8}$.  We say that two $(N,N)$-sectors, $\Gamma_{N,N}^{\omega}$ and $\Gamma_{N,N}^{\omega'}$ are related when $|\omega-\omega'| \leq 2^{-N+8}$.    

We can deduce a near-optimal $L^2$-based bilinear adjoint restriction theorem for related $(N,k)$-sectors from results already in the literature.  Namely, one may directly apply the bilinear restriction method from \cite{TaoParab} (which was quickly observed to apply to conic surfaces) and conic rescaling, or else directly apply the results of \cite{Candy} to obtain the following.  

\begin{theorem} [\cite{Candy, VargasLee_kflat, TaoParab}] \label{T:bilin cone}
Let $C \leq k < N$, let $\Gamma_{N,k}^{\omega_1}$ and $\Gamma_{N,k}^{\omega_2}$ be related $(N,k)$-sectors, and let $f_1,f_2$ be $L^2$ functions supported on $\Gamma_{N,k}^{\omega_1},\Gamma_{N,k}^{\omega_2}$, respectively.  Then
\begin{equation} \label{E:bilin cone}
\|\scriptE f_1 \scriptE f_2\|_{L^{q/2}} \lesssim 2^{-(N-k)(\frac 6q-1)}\|f_1\|_2\|f_2\|_2, \qquad q > \tfrac{10}3.
\end{equation}
\end{theorem}

We state our bilinear-to-linear deduction in slightly more general terms than we need in this section in order to facilitate later arguments.  

\begin{lemma}\label{bilinear-to-linear}
Let $3 < q < 4$, $(\tfrac q2)' \leq p \leq q $, and $s\leq p$.  Assume that  $\scriptR_0^*(p \to q)$ holds and that for $C \leq k < N$, 
\begin{align}\label{ex4}
\|\fE f_1 \fE f_2\|_{q/2} \lesssim 2^{-(N-k)\alpha}\|f_1\|_s\|f_2\|_s,
\end{align}
whenever $f_1$ and $f_2$ are supported in related $(N,k)$-sectors. If $\alpha \geq \tfrac2s-\tfrac2p$, $\alpha > 0$, and either $\alpha \neq \tfrac2s-\tfrac2q$ or $p < q$, then
\begin{align*}
\|\fE f_N\|_q \lesssim \|f_N\|_p,
\end{align*}
for all measurable functions $f_N$ satisfying $|f_N| \sim \one_{\Omega_N}$, for some $\Omega_N \subseteq \Gamma_N$.  
\end{lemma}

Lemma~\ref{bilinear-to-linear} implies a restricted strong type inequality for extension from the $\Gamma_N$, which can be interpolated to yield strong type inequalities since we work with exponents obeying $q \geq p$.   

\begin{proof}[Proof of Lemma~\ref{bilinear-to-linear}]
We may choose $O(2^{-k})$-separated collections $\scriptD_{N,k}\subseteq \mathbb{S}^1$, 
$C \leq k \leq N$,  such that whenever $(\tau,\xi),(\tau',\xi') \in \Gamma_N$, there exists a pair of related $(N,k)$-sectors $\Gamma_{N,k}^{\omega} \ni (\tau,\xi)$ and $\Gamma_{N,k}^{\omega'}\ni (\tau',\xi')$, with $\omega,\omega' \in \scriptD_{N,k}$.  Here $k=N$ if $|\tfrac{\xi}{|\xi|}-\tfrac{\xi'}{|\xi'|}| \lesssim 2^{-N}$, and $2^{-k} \sim |\tfrac{\xi}{|\xi|}-\tfrac{\xi'}{|\xi'|}|$, otherwise.  We will abuse notation by saying that for $\omega,\omega' \in \scriptD_{N,k}$, $\omega \sim \omega'$ if $\Gamma_{N,k}^\omega \sim \Gamma_{N,k}^{\omega'}$.  Thus we may decompose
\begin{equation} \label{E:Whitney}
\Gamma_N \times \Gamma_N:= \bigcup_{k=C}^N \bigcup_{\omega \sim \omega' \in \scriptD_{N,k}} \Gamma_{N,k}^\omega \times  \Gamma_{N,k}^{\omega'}.  
\end{equation}
We will later use the geometric property that each $\Gamma_{N,k}^{\omega}$ is contained in a parallelepiped $P_{N,k}^\omega$, such that the sumsets $P_{N,k}^\omega+P_{N,k}^{\omega'}$ are finitely overlapping as the pair $\omega \sim \omega' \in \scriptD_{N,k}$ varies.  

Let $f_N$ be a measurable function with $|f_N| \sim \one_{\Omega_N}$, for some subset $\Omega_N \subseteq \Gamma_N$.  Using the decomposition \eqref{E:Whitney} to make a partition of unity, we have
\begin{align*}
\|\fE \one_{\Omega_N}\|_q^2 &= \|(\fE \one_{\Omega_N})^2\|_{q/2}\\
&\leq \bigg\|\sum_{\omega \sim \omega' \in \sD_{N,N}}\scriptE f_{N,N}^\omega \scriptE f_{N,N}^{\omega'}\bigg\|_{q/2} + \bigg\|\sum_{k=C}^{N-1}\sum_{\omega\sim\omega'\in\sD_{N,k}} \scriptE f_{N,k}^\omega \scriptE f_{N,k}^{\omega'}\bigg\|_{q/2}=: I_1 + I_2,
\end{align*}
where the $f_{N,k}^\omega$ are measurable functions supported on the $\Gamma_{N,k}^\omega$ with $|f_{N,k}^\omega| \lesssim |f_N|$.  

We begin with the first term.  By almost orthogonality and the finite overlap of sumsets, the Cauchy--Schwarz inequality, the hypothesis that $\scriptR_0^*(p \to q)$ holds, and the fact that $q \geq p$, we have
\begin{align*}
I_1 
&\lesssim \bigg(\sum_{\omega \sim \omega'\in\sD_{N,N}}\|\fE f_{N,N}^\omega\fE f_{N,N}^{\omega'}\|_{q/2}^{q/2}\bigg)^{2/q}
\leq \bigg(\sum_{\omega \sim \omega'\in\sD_{N,N}}\|\fE f_{N,N}^\omega\|_q^{q/2}\|\fE f_{N,N}^{\omega'}\|_q^{q/2}\bigg)^{2/q}\\
&\lesssim \bigg(\sum_{\omega \sim \omega'\in\sD_{N,N}}\|f_{N,N}^\omega\|_p^{ q/2}\|f_{N,N}^{\omega'}\|_p^{ q/2}\bigg)^{2/q}
\lesssim \bigg(\sum_{\omega \in \sD_{N,N}} \|f_N \one_{\Gamma_{N,N}^\omega}\|_p^q\bigg)^{2/q}
\lesssim \|f_N\|_p^2.
\end{align*}

Now we turn to the second term.  Let $\Omega_{N,k}^\omega:=\Omega_N \cap \Gamma_{N,k}^\omega$. By the triangle inequality, the Tao--Vargas--Vega orthogonality lemma \cite[Lemma 6.1]{TVV}, and the aforementioned finite overlap property of sumsets, and then \eqref{ex4} and some standard reindexing,
\begin{align*}
    I_2 
    \lesssim \sum_{k=C}^{N-C} \left(\sum_{\omega \sim \omega' \in \scriptD_{N,k}} \|\scriptE f_{N,k}^\omega \scriptE f_{N,k}^{\omega'}\|_{\frac q2}^{\frac q2}\right)^{\frac 2q}
    \lesssim \sum_{k=C}^{N-C} 2^{-(N-k)\alpha}\bigl(\sum_{\omega \in \scriptD_{N,k}} \sigma(\Omega_{N,k}^\omega)^{\frac qs} \bigr)^{\frac 2q}.
\end{align*}
Thus by H\"older's inequality and the estimates 
$$
\sigma(\Omega_{N,k}^\omega) \leq \min\{\sigma(\Omega_N),\sigma(\Gamma_{N,k}^\omega)\} \qtq{and} \sigma(\Gamma_{N,k}^\omega) \sim 2^{N-k}, 
$$
we see that
\begin{align*}
    I_2 &\lesssim \sum_{j=C}^{N-C} 2^{-j\alpha}\min\{2^{j(\frac 2s-\frac2q)},\sigma(\Omega_N)^{\frac2s-\frac2q}\}|\Omega_N|^{\frac 2q} \\
    &\leq  \sum_{j=C}^{\log_2(\sigma(\Omega_N))} 2^{j(\frac2s-\frac2q-\alpha)}\sigma(\Omega_N)^{\frac 2q} + \sum_{j=\max\{C,\log_2(\sigma(\Omega_N))\}}^{N-C} 2^{-j\alpha}\sigma(\Omega_N)^{\frac 2s}=:I_2'+I_2''.
\end{align*}
When $\sigma(\Omega_N) \leq 1$, $I_2'=0$ and $I_2'' \sim \sigma(\Omega_N)^{\frac 2s} \leq \sigma(\Omega_N)^{\frac 2p}$, since $s\leq p$.  
When $\sigma(\Omega_N) \geq 1$, $I_2'' \sim \sigma(\Omega_N)^{\frac 2s-\alpha} \leq \sigma(\Omega_N)^{\frac 2p}$.  If, in addition, $\tfrac2s-\tfrac2q-\alpha < 0$, $I_2' \sim \sigma(\Omega_N)^{\frac 2q} \leq \sigma(\Omega_N)^{\frac 2p}$.  Meanwhile, if $\tfrac2s-\tfrac2q -\alpha > 0$, $I_2' \sim \sigma(\Omega_N)^{\frac 2s-\alpha} \leq \sigma(\Omega_N)^{\frac 2p}$.  Finally, if $\alpha = \frac2s-\frac2q$ and $p<q$, then $I_2' \sim \log(\sigma(\Omega_N))\sigma(\Omega_N)^{\frac2q} \lesssim \sigma(\Omega_N)^{\frac2p}$.

In any case, combining our estimates for $I_1$ and $I_2$ gives $\|\fE f_N\|_q^2 \lesssim \sigma(\Omega_N)^{\frac2p}$, completing the proof of the lemma.  
\end{proof}

Theorem \ref{T:bilin cone}, Theorem \ref{linear est zero region}, Lemma~\ref{L:zero to thin sector}, Lemma \ref{bilinear-to-linear} (with $q \geq p>2$, $q \geq 2p'$, and $s=2$), and real interpolation together  imply that $\|\scriptE f\|_q \lesssim \|f\|_p$ for all $f \in C_c^\infty(\Gamma_N)$ when $q > 10/3$ and $(\tfrac q2)' \leq p \leq q$.  Thus the proof of Proposition~\ref{P:annuli bilin range} is complete.

\section{Reduction to bounds on $\Gamma_0$ via decoupling} \label{S:decoupling}

In the previous section we showed how to deduce bounds for extension from the dyadic annuli $\Gamma_N$ from those for extension from $\Gamma_0$ by using the bilinear adjoint restriction inequality \eqref{E:bilin cone}.  This approach is limited, since we do not currently know any such result with $q \leq \tfrac{10}3$.  In this section, we will use the conic decoupling theorem of Bourgain--Demeter to obtain new bounds for extension from $\Gamma_N$, conditional on further improvements to $\scriptR_0^*(p \to q)$.  
The entirety of this section will be devoted to a proof of the following result.  

\begin{proposition}\label{P:annuli via decoup}
Suppose that $\scriptR_0^*((\frac{q_0}{2})'\rightarrow q_0)$ holds for some $q_0 < \tfrac{10}3$.  
 Then for $(p,q)$ obeying $(\tfrac q2)' \leq p \leq q$ and 
\begin{equation} \label{E:conditional p}
\frac{1}{p} > \frac 25 \cdot \frac{1/q-3/10}{1/q_0-3/10}+\frac1{10},
\end{equation}
we have
\begin{align}\label{bounds on annuli}
\|\fE f_N\|_q \lesssim \|f_N\|_{p},
\end{align}
for all $f_N \in L^{p}(\Gamma_N)$, with bounds uniform in $N$.
\end{proposition}

Our main tool in the proof of Proposition~\ref{P:annuli via decoup} is the following consequence of Bourgain--Demeter's decoupling theorem for the cone.  

\begin{proposition}\label{P:sectors via decoup}
Suppose that $\scriptR_0^*(p \rightarrow q)$ holds for some $p \geq (\frac{q}{2})'$, $q \leq 4$.  Then 
\begin{align*}
\|\fE f \|_q \lesssim_\varepsilon 2^{(N-k)(\frac{1}{2}-\frac{1}{p}+\varepsilon)}\|f\|_p
\end{align*}
for all functions $f$ supported in an $(N,k)$-sector and all $\varepsilon > 0$.
\end{proposition}

\begin{proof}[Proof of Propostion~\ref{P:sectors via decoup}]
Let $\kappa$ be an $(N,k)$-sector, let $f_\kappa$ be supported in $\kappa$, and let $\fP$ be a partition of $\kappa$ into $(N,N)$-sectors.  The estimate $\scriptR_0^*(p\rightarrow q)$ and Lemma \ref{L:zero to thin sector} imply that
\begin{align}\label{thin sector est}
    \|\fE f_\theta\|_q \lesssim \|f_\theta\|_p
\end{align}
for all $f_\theta$ supported in $\theta \in \fP$.  In particular, if $N - k \lesssim 1$, then $\#\fP \lesssim 1$ and the required estimate is a consequence of the triangle inequality and \eqref{thin sector est}.  We may assume, therefore, that $N - k \geq C$ for some sufficiently large constant $C$.

We proceed by rescaling extension estimates on $\kappa$ to those on a nearly conic set of angular width $1$ in the region $|\xi| \sim 1$, where Bourgain--Demeter's conic decoupling theorem can be directly applied.  We may assume by rotational symmetry that
\begin{align}\label{def of kappa}
    \kappa = \{(\llangle \xi \rrangle, \xi) : 2^N \leq |\xi| \leq 2^{N+1},~\angle(\xi,(1,0)) \leq 2^{-k-1}\}.
\end{align}
Thus, $\kappa$ lies in an $O(2^{-N})$-neighborhood of the conic sector 
\begin{align*}
\kappa_c := \{(|\xi|,\xi) : 2^N \leq |\xi| \leq 2^{N+1},~\angle(\xi,(1,0)) \leq 2^{-k-1}\}.
\end{align*}
Let $D$ be the conic dilation $D(\tau,\xi) := 2^{-N}(\tau,\xi)$.  Then $D(\kappa_c)$ is a conic sector of angular width $2^{-k}$ in $C_0 := \{(|\xi|,\xi) : 1 \leq |\xi| \leq 2\}$ that contains the point $(1,1,0)$.  Let $L$ be the linear map satisfying
\begin{align*}
L(0,0,1) &= 2^{k}(0,0,1),\\
L(1,1,0) &= (1,1,0),\\
L(-1,1,0) &= 2^{2k}(-1,1,0).
\end{align*}
Geometrically, the vectors $(0,0,1)$, $(1,1,0)$, and $(-1,1,0)$ are respectively ``angularly tangent," ``radially tangent," and normal to $C_0$ at the point $(1,1,0)$.  The map $L$ preserves the cone and expands $D(\kappa_c)$ to angular width $1$.  Now, set $M = LD$ and $\delta = C'2^{2(k-N)}$, where $C'$ is a constant.  If $C'$ is sufficiently large, then $M(\kappa)$ lies in the $\delta$-neighborhood of a conic frustum $\tilde{C}_0$, a slight enlargement of $C_0$.  Let $\d M_\ast\sigma$ be the pushforward measure on $M(\Gamma)$, given by
\begin{align*}
\int_{M(\Gamma)} g\, \d M_\ast \sigma := \int_{\Gamma}g \circ M\, \d\sigma,
\end{align*}
and let $\fE^M g := (g\d M_\ast\sigma)^\vee$.  Let $\tilde{\fP}$ be a partition of the $\delta$-neighborhood of $\tilde{C}_0$ into sectors $\Delta$ of angular width $\delta^{1/2}$ and thickness $\delta$.  By conic decoupling, see \cite[Theorem~1.2]{BDcone}, the inequality
\begin{align}\label{dec}
\|\fE^M g\|_q \lesssim_\varepsilon \delta^{-\varepsilon}\bigg(\sum_{\Delta \in \tilde{\fP}}\|\fE^M(g\one_{\Delta'})\|_q^2\bigg)^\frac{1}{2}
\end{align}
holds for all $g$ supported in $M(\kappa)$, where $\Delta' := \Delta \cap M(\kappa)$.  We claim that every $\Delta \in \tilde{\fP}$ obeys the bound
\begin{align}\label{count}
\#\{\theta \in \fP : \theta \cap M^{-1}(\Delta') \neq \emptyset\} \lesssim 1.
\end{align}
Then, taking $g = f \circ M^{-1}$ in \eqref{dec}, rescaling, and applying \eqref{thin sector est} and H\"older's inequality, we get
\begin{align*}
\|\fE f\|_q &\lesssim_\varepsilon \delta^{-\varepsilon}\bigg(\sum_{\Delta \in \tilde{\fP}}\|\fE(f\one_{M^{-1}(\Delta')})\|_q^2\bigg)^\frac{1}{2}\\
&\lesssim \delta^{-\varepsilon}\bigg(\sum_{\Delta \in \tilde{\fP}}\sum_{\substack{\theta \in \fP :\\  \theta \cap M^{-1}(\Delta') \neq \emptyset}}\|f\|_{L^p(\theta \cap M^{-1}(\Delta'))}^2\bigg)^\frac{1}{2}\\
&\lesssim 2^{(N-k)(\frac{1}{2}-\frac{1}{p}+2\varepsilon)}\|f\|_p.
\end{align*}
Since $\varepsilon$ is arbitrary, the proof is complete modulo the claim \eqref{count}.

To begin the proof of \eqref{count}, we record the following notation: The angular separation of $\zeta,\zeta' \in \R^3$ is defined as
\begin{align*}
    \dist_{\ang}(\zeta,\zeta') := \left\vert \frac{(\zeta_2,\zeta_3)}{|(\zeta_2,\zeta_3)|} - \frac{(\zeta_2',\zeta_3')}{|(\zeta_2',\zeta_3')|}\right\vert.
\end{align*}
Now, fix $\Delta \in \tilde{\fP}$ and let $n := \#\{\theta \in \fP : \theta \cap M^{-1}(\Delta') \neq \emptyset\}$. We need to show that $n \lesssim 1$, so we may assume that $n \geq 3$. Then there exist $\zeta,\zeta' \in \kappa \cap M^{-1}(\Delta')$ such that $\dist_{\ang}(\zeta,\zeta') \gtrsim n2^{-N}$.  Since $\Delta'$ has angular width $O(2^{k-N})$, it suffices to show that $\dist_{\ang}(M(\zeta),M(\zeta')) \gtrsim 2^k\dist_{\ang}(\zeta,\zeta')$.  Toward that end, it will be convenient to understand how $M$ transforms the polar coordinates $(\llangle \xi \rrangle, \xi) =: (\llangle r \rrangle, r\cos\nu, r \sin\nu)$, where $\llangle r \rrangle := \sqrt{r^2-1}$.  We compute that $M(\llangle r \rrangle, r\cos\nu,r\sin\nu) = 2^{-N-1}(m_i(r,\nu))_{i=1}^3$, where
\begin{align*}
    m_1(r,\nu) &:= (1+2^{2k})\llangle r \rrangle + (1-2^{2k})r\cos\nu,\\
    m_2(r,\nu) &:= (1-2^{2k})\llangle r \rrangle + (1+2^{2k})r\cos\nu,\\
    m_3(r,\nu) &:= 2^{k+1}r\sin\nu.
\end{align*}
The polar angle associated to $M(\llangle r \rrangle, r\cos\nu, r\sin\nu)$ is
\begin{align*}
A(r,\nu) := \arctan\bigg(\frac{m_3(r,\nu)}{m_2(r,\nu)}\bigg).
\end{align*}
Thus, letting $\zeta =: (\llangle r \rrangle, r\cos\nu, r\sin\nu)$ and $\zeta' =: (\llangle r' \rrangle, r'\cos\nu', r'\sin\nu')$, we have
\begin{align}\label{angular separation}
    \dist_{\ang}(M(\zeta),M(\zeta')) \sim |A(r,\nu) - A(r',\nu')|.
\end{align}
Since $\zeta \in \kappa$, we know that $2^N \leq r \leq 2^{N+1}$ and $|\nu| \leq 2^{-k-1}$ by \eqref{def of kappa}.  Consequently, one easily checks that $m_2(r,\nu) \sim 2^N$ and $|m_3(r,\nu)| \lesssim 2^N$; the same bounds hold for $m_2(r',\nu')$ and $m_3(r',\nu')$.  Arguments using the mean value theorem and the preceding estimates show that
\begin{align*}
    |A(r,\nu) - A(r,\nu')| \geq |\nu - \nu'|\inf_{|\varphi| \leq 2^{-k-1}}|\partial_2 A(r,\varphi)| \gtrsim 2^k\dist_{\ang}(\zeta,\zeta')
\end{align*}
and
\begin{align*}
    |A(r,\nu')-A(r',\nu')| \leq |r-r'|\sup_{2^N \leq s \leq 2^{N+1}}|\partial_1 A(s,\nu')| \lesssim 2^N 2^{2k-3N} \lesssim 2^{-C}2^k\dist_{\ang}(\zeta,\zeta').
\end{align*}
Thus, if $C$ is sufficiently large, then $|A(r,\nu)-A(r',\nu')| \gtrsim 2^k\dist_{\ang}(\zeta,\zeta')$ by the triangle inequality.  Plugging this estimate into \eqref{angular separation} completes the proof.
\end{proof}

We are now ready to prove Proposition \ref{P:annuli via decoup}.  By the hypothesis $\scriptR_0^*((\frac{q_0}{2})' \rightarrow q_0)$, Proposition \ref{P:sectors via decoup}, and the Cauchy--Schwarz inequality, we have
\begin{align*}
    \|\fE f_1 \fE f_2\|_{q_0/2} \lesssim_\varepsilon 2^{-(N-k)(1-\frac{4}{q_0}-2\varepsilon)}\|f_1\|_{p_0}\|f_2\|_{p_0}
\end{align*}
for all functions $f_1$, $f_2$ supported in $(N,k)$-sectors.  Given $q_1 > \frac{10}{3}$, we also have
\begin{align*}
    \|\fE f_1 \fE f_2\|_{q_1/2} \lesssim 2^{-(N-k)(\frac{6}{q_1}-1)}\|f_1\|_2\|f_2\|_2
\end{align*}
by Theorem \ref{T:bilin cone}, provided $f_1$ and $f_2$ are supported in related $(N,k)$-sectors.  Interpolating these estimates, we see that
\begin{align}\label{interpolated est}
    \|\fE f_1 \fE f_2\|_{q_t/2} \lesssim_\varepsilon 2^{-(N-k)\alpha_t}\|f_1\|_{s_t}\|f_2\|_{s_t},
\end{align}
for $0 \leq t \leq 1$, where
\begin{align*}
    \left(\frac{1}{s_t},\frac{1}{q_t}\right) &:= (1-t)\left(1-\frac{2}{q_0}, \frac{1}{q_0}\right) + t\left(\frac{1}{2},\frac{1}{q_1}\right),\\
    \alpha_t &:= (1-t)\left(1-\frac{4}{q_0}-2\varepsilon\right) + t\left(\frac{6}{q_1}-1\right).
\end{align*}
We may apply Lemma~\ref{bilinear-to-linear} to obtain uniform restricted weak type $L^p \to L^{q_t}$ bounds on dyadic annuli as long as $(\frac{q_t}{2})' \leq p \leq q_t$ and
$$
\frac1p > \frac1{s_t}-\frac{\alpha_t}2, 
$$
or, equivalently, after a bit of arithmetic, if
\begin{equation} \label{E:temp conditional p}
\frac1p > \left(\frac3{q_1}-\frac12+\eps\right)(1-t)+1-\frac3{q_1}.
\end{equation}
Sending $q_1 \searrow \tfrac{10}3$ and $\eps \searrow 0$, and substituting $1-t = \tfrac{1/{q_t}-1/q_1}{1/q_0-1/q_1}$ in \eqref{E:temp conditional p} yields \eqref{E:conditional p}.  Having proved restricted weak type bounds in the claimed region, real interpolation completes the proof of Proposition \ref{P:annuli via decoup}.

\section{Summing the bounds on annuli} \label{S:annuli to global}

The purpose of this section is complete the proof of Theorem~\ref{T:pos} by proving that uniform bounds for the extension from dyadic annuli imply global bounds on $\scriptE$.  Let $\scriptR^*_{ann}(p \to q)$ denote the statement that for all $N \geq 1$ and measurable $f_N$ supported on $\Gamma_N$,
$$
\|\scriptE f_N\|_q \lesssim \|f_N\|_p.
$$
We will spend the majority of this section proving the following.

\begin{lemma} \label{L:annular Strichartz}
If $\scriptR^*_{ann}(p_0 \to q_0)$ holds for some $(\tfrac{q_0}2)' \leq p_0 \leq q_0$, then $\scriptR^*(p \to q)$ holds for all $q > q_0$ and $p' = \tfrac{p_0'}{q_0} q$. 
\end{lemma}

\begin{lemma} \label{L:reduce annuli p=q}
If $\scriptR^*_{ann}(q_0 \to q_0)$ holds for some $3<q_0<4$, then $\scriptR^*(q \to q)$ holds for all $q_0 < q < 4$.  
\end{lemma}

Before proving the lemmas in detail, we note that applying them in conjunction with Propositions~\ref{P:annuli bilin range} and~\ref{P:annuli via decoup} completes the proof of Theorem~\ref{T:pos}.  

We will prove Lemmas~\ref{L:annular Strichartz} and~\ref{L:reduce annuli p=q} by proving that the hypotheses imply a bilinear extension estimate between annuli:  \begin{equation} \label{E:bilin annuli}
\|\scriptE f_{N_1} \scriptE f_{N_2}\|_{\frac q2} \lesssim 2^{-c_0|N_1-N_2|}\|f_{N_1}\|_p\|f_{N_2}\|_p,
\end{equation}
for some $c_0>0$, and measurable functions $|f_{N_j}| \sim \one_{\Omega_{N_j}}$, $\Omega_{N_j} \subseteq \Gamma_{N_j}$, $j=1,2$.  Indeed, assuming validity of such an estimate, for any $|f| \sim \one_{\Omega}$, by the triangle inequality and $q \leq 4$,
\begin{align*}
\|\scriptE f\|_q^q 
&\lesssim \|f\|_p^q + \|\sum_{N \geq C}\scriptE f_N\|_q^q \lesssim \|f\|_p^q + \sum_{N_1\geq N_2 \geq N_3 \geq N_4 \geq C} \|\prod_{i=1}^4 \scriptE f_{N_i}\|_{\frac q4}^{\frac q4}\\
&\lesssim \|f\|_p^q + \sum_{N_1\geq N_2 \geq N_3 \geq N_4 \geq 1} 2^{-\frac{qc_0}4|N_1-N_4|}\prod_{i=1}^4\|f\|_{L^p(\Gamma_{N_i})}^{\frac q4}
\lesssim \sum_{N \geq 0}\|f\|_{L^p(\Gamma_N)}^q \lesssim \|f\|_{L^p}^q.
\end{align*}
Real interpolation leads to strong-type bounds. 

\begin{proof}[Proof of Lemma~\ref{L:annular Strichartz}]
The Strichartz inequality \eqref{mixednormStr} implies that
\begin{equation}\label{E:r,s strichartz}
\begin{gathered}
\|\scriptE f\|_{L^r_tL^s_x} \lesssim \|\hjp{\xi}^{\frac1r-\frac1s}f\|_{L^2(\Gamma;\d\sigma)}, \\ 
2 \leq r,s; \quad s < \infty;\quad \frac 2r + \frac{2p_0'}{(q_0-p_0')s} = \frac{p_0'}{q_0-p_0'}.
\end{gathered}
\end{equation}
As \eqref{E:r,s strichartz} implies boundedness of $\scriptE$ in the range $p=2$, $4 \leq q \leq 6$, we may assume henceforth that $p_0>2$.

Let $q_2 = 2\tfrac{q_0}{p_0'}$, and choose some $r_0, s_0,r_1,s_1$ obeying \eqref{E:r,s strichartz}, $r_0 < q_2 < s_0$, and 
$$
\frac1{q_2} = \frac12(\frac1{r_0}+\frac1{r_1}) = \frac12(\frac1{s_0}+\frac1{s_1}).
$$
By the Cauchy--Schwarz inequality, for any $1 \leq N_1 \leq N_2$, we have the bilinear estimate
\begin{align*}
    \|\scriptE f_{N_1}\scriptE f_{N_2}\|_{L^{\frac{q_2}2}} &\lesssim \|\scriptE f_{N_1}\|_{L^{r_0}_tL^{s_0}_x} \|\scriptE f_{N_2}\|_{L^{r_1}_tL^{s_1}_x} \lesssim 2^{N_1(\tfrac1{r_0}-\tfrac1{s_0})}2^{N_2(\tfrac1{r_1}-\tfrac1{s_1})}\|f_{N_1}\|_2\|f_{N_2}\|_2 \\
    &= 2^{-(\tfrac1{r_0}-\tfrac1{s_0})|N_1-N_2|}\|f_{N_1}\|_2\|f_{N_2}\|_2.
\end{align*}
Inequality \eqref{E:bilin annuli} follows by interpolation with the consequence 
$$
\|\scriptE f_{N_1} \scriptE f_{N_2}\|_{\frac{q_0}2} \lesssim \|f_{N_1}\|_{p_0}\|f_{N_2}\|_{p_0}
$$
of our hypothesis.
\end{proof}

\begin{proof}[Proof of Lemma~\ref{L:reduce annuli p=q}]
We will prove a bilinear estimate between annuli as in \eqref{E:bilin annuli}, with $p=q$ and $C \leq N_1 \leq N_2-C$ fixed.  To do so, we will use three different bilinear extension estimates between sectors at different scales.  

It is convenient to modify our Whitney decomposition slightly from earlier, though we will continue to use the convention that $|\tfrac{\xi}{|\xi|}-e_1| < c$ for all $\xi \in \Gamma_N$, for some sufficiently small $c$.  For $C \leq k \leq N_1$, let $\scriptD_k$ denote a $2^{-k}$-separated subset of $\mathbb{S}^1$.  For $\omega,\omega' \in \scriptD_k$ and $k < N_1-C$, we say that $\omega \sim \omega'$ if $2^{-k+C} \leq |\omega-\omega'| \leq 2^{-k+2C}$. Meanwhile, for $N_1-C \leq k \leq N_1$, we say that $\omega \sim \omega'$ if $|\omega - \omega'| \leq 2^{-N_1+2C}$. Thus for $\xi_1 \in \Gamma_{N_1}$ and $\xi_2 \in \Gamma_{N_2}$, there is at least one and at most a bounded number of triples $(k,\omega, \omega')$ with $\omega \sim \omega' \in \scriptD_k$ and $\xi_1 \in \Gamma_{N_1,k}^\omega$ and $\xi_2 \in \Gamma_{N_2,k}^{\omega'}$.  

By the hypothesis that $\scriptR^*_{ann}(q_0 \to q_0)$ holds, interpolation, and the Cauchy--Schwarz inequality, for $k \geq C$ and $\omega,\omega' \in \scriptD_k$, we have
\begin{equation} \label{E:CS p=q}
\|\scriptE f_{N_1,N_1}^\omega \,\scriptE f_{N_2,N_2}^{\omega'}\|_{\frac{q}2} \lesssim \|f_{N_1,N_1}^\omega\|_{q}\|f_{N_2,N_2}^{\omega'}\|_{q}, \qquad q_0 \leq q < 4.
\end{equation}
By Theorem~1.4 of \cite{Candy} (see also \cite[Theorem~1.10]{Candy}), if $0 \leq k < N_1-C$, $\omega \sim \omega' \in \scriptD_k$, and $\tfrac{10}3 < q_1 < 4$, then 
\begin{equation} \label{E:Candy scales p=q}
\|\scriptE f_{N_1,k}^\omega\,\scriptE f_{N_2,k}^{\omega'}\|_{\frac{q_1}2} \lesssim 2^{-(N_1+N_2-2k)(\frac3{q_1}-\frac12)}2^{-(\frac 32-\frac 5{q_1})|N_1-N_2|}\|f_{N_1,k}^\omega\|_2\|f_{N_2,k}^{\omega'}\|_2.  
\end{equation}
Finally, if $k \geq N_1-C$ and $\omega \sim \omega' \in \scriptD_k$, we claim that
\begin{equation} \label{E:L4 bilin}
\|\scriptE f_{N_1,k}^\omega \, \scriptE f_{N_2,k}^{\omega'}\|_2 \lesssim 2^{-\frac14|N_1-N_2|}\|f_{N_1,k}^\omega\|_4\|f_{N_2,k}^{\omega'}\|_4.
\end{equation}

We now turn to the details of \eqref{E:L4 bilin}, which follow a well-established route.  Let $(\hjp\xi,\xi) \in \Gamma_{N_1,k}^\omega$ and $(\hjp\eta,\eta) \in \Gamma_{N_2,k}^{\omega'}$.  The coordinate change $\zeta = (\hjp\xi+\hjp\eta,\xi+\eta)$, $\beta = \xi^\perp$ (perpendicular direction taken with respect to $\omega$) is finite-to-one, and has Jacobian determinant $|\tfrac{\partial(\zeta,\beta)}{\partial(\xi,\eta)}| \sim 2^{-2N_1}$.  By Plancherel's identity, the change of variables formula (recall that we integrate with respect to $\d\sigma$), and H\"older's inequality ($\beta$ varies over an interval of length at most 1), the right-hand side of \eqref{E:L4 bilin} is bounded by
\begin{align*}
\|\widehat{\scriptE f_{N_1,k}^\omega}\,\widehat{\scriptE f_{N_2,k}^{\omega'}}\|_2
\lesssim \bigl(\iint |f_{N_1,k}^\omega(\hjp\xi,\xi)\,f_{N_2,k}^{\omega'}(\hjp\eta,\eta)\tfrac1{\hjp\xi\hjp\eta} |\tfrac{\partial(\xi,\eta)}{\partial(\zeta,\beta)}||^2\,\d\beta\,\d\zeta\bigr)^{\frac12}.  
\end{align*}
Changing variables back, estimating the various roughly constant terms that have arisen, and using H\"older's inequality again, the right-hand side of the preceding inequality is bounded by
\begin{equation}\label{E:L4 L2 p=q}
2^{N_1} 2^{-\frac12(N_1+N_2)}\sigma(\Gamma_{N_1,k}^\omega)^{\frac14}\sigma(\Gamma_{N_2,k}^{\omega'})^{\frac14}\|f_{N_1,k}^\omega\|_4\|f_{N_2,k}^{\omega'}\|_4.
\end{equation}
Since $\sigma(\Gamma_{N_1,k}^\omega) \lesssim 1$, while $\sigma(\Gamma_{N_2,k}^{\omega'}) \lesssim 2^{N_2-N_1}$, inequality \eqref{E:L4 L2 p=q} implies \eqref{E:L4 bilin}.  

Interpolating \eqref{E:CS p=q} and \eqref{E:Candy scales p=q}, yields, for all $\omega \sim\omega' \in \scriptD_k$, $0 \leq k < N_1-C$,
\begin{equation} \label{E:bilin scales sep p=q}
\|\scriptE f_{N_1,k}^\omega \scriptE f_{N_2,k}^{\omega'}\|_{\frac q2} \lesssim 2^{-(N_1+N_2-2k)\frac\alpha2}2^{-\delta|N_1-N_2|}\|f_{N_1,k}^\omega\|_s\|f_{N_2,k}^{\omega'}\|_s,
\end{equation}
for all $q_0<q<4$, some $s<q$, some $\alpha > \tfrac 2s-\tfrac2q$, and some $\delta > 0$.  
Interpolating \eqref{E:L4 bilin} and \eqref{E:CS p=q} (taking $q=q_0$ in the latter) yields, for $k = N_1+C$ and  $\omega \sim \omega' \in \scriptD_{N_1,k}$,
\begin{equation} \label{E:bilin scales p=q}
\|\scriptE f_{N_1,k}^\omega \scriptE f_{N_2,N_1}^{\omega'}\|_{\frac q2} \lesssim 2^{-\delta|N_1-N_2|}\|f_{N_1,N_1}^\omega\|_q\|f_{N_2,N_1}^{\omega'}\|_q, \qquad q_0 < q < 4.
\end{equation}

We adapt the bilinear to linear argument of Tao--Vargas--Vega \cite{TVV}. Namely, if $|f_{N_j}| \sim \one_{\Omega_{N_j}},$ $j=1,2$, then using a partition of unity and almost orthogonality; using \eqref{E:bilin scales sep p=q},  \eqref{E:bilin scales p=q}, the Cauchy--Schwarz inequality, and reindexing; and  finally summing as in the proof of Lemma~\ref{bilinear-to-linear},
\begin{align*}
    &\|\scriptE f_{N_1}\scriptE f_{N_2}\|_{\frac q2} 
    \lesssim \sum_{k=C}^{N_1} \bigl(\sum_{\omega \sim \omega' \in \scriptD_k}\|\scriptE f_{N_1,k}^\omega \scriptE f_{N_2,k}^{\omega'}\|_{\frac q2}^{\frac q2}\bigr)^{\frac 2q} \\
    &\qquad\lesssim 2^{-\delta|N_1-N_2|}\left(  \prod_{j=1}^2\bigl(\sum_{\omega \in \scriptD_{N_1}} \|f_{N_j,k}^\omega\|_q^q\bigr)^{\frac 1q}
    + \prod_{j=1}^2\bigl(\sum_{k=C}^{N_1} 2^{-(N_j-k)\alpha}\bigl(\sum_{\omega \in \scriptD_k}\|f_{N_j,k}^\omega\|_s^q\bigr)^{\frac1q}\right)\\
    &\qquad \lesssim 2^{-\delta|N_1-N_2|}\|f_{N_1}\|_q\|f_{N_2}\|_q.
\end{align*}

\end{proof}

\section*{Acknowledgements} This work was supported by NSF DMS-1653264 and the EPSRC New Investigator Award ``Sharp Fourier Restriction Theory'', grant no.\@ EP/T001364/1.

\end{document}